\documentclass[11pt]{amsart}

\usepackage{amsmath,amssymb,amsthm}
\usepackage{graphicx}
\usepackage{tikz}
\usetikzlibrary{arrows.meta}
\usepackage[hidelinks]{hyperref}

\newcommand{\fitwidth}[1]{%
  \resizebox{\ifdim\width>\linewidth\linewidth\else\width\fi}{!}{#1}}

\newcommand{\N}{\mathbb{N}}
\newcommand{\R}{\mathbb{R}}
\newcommand{\Z}{\mathbb{Z}}
\newcommand{\Q}{\mathbb{Q}}
\newcommand{\Qtwo}{\mathbb{Q}_{2}}
\newcommand{\Qbar}{\overline{\mathbb{Q}}}
\newcommand{\nrm}[1]{\lVert #1\rVert}   
\newcommand{\cx}{P}                     
\newcommand{\Kc}{K}                     

\theoremstyle{plain}
\newtheorem{theorem}{Theorem}[section]
\newtheorem{theoremintro}{Theorem}

\newtheorem{proposition}[theorem]{Proposition}
\newtheorem{lemma}[theorem]{Lemma}
\newtheorem{corollary}[theorem]{Corollary}
\newtheorem{conjecture}[theorem]{Conjecture}

\theoremstyle{definition}
\newtheorem{definition}[theorem]{Definition}

\theoremstyle{remark}
\newtheorem{remark}[theorem]{Remark}

\begin{document}

\title[The minimal word of the rational base $3/2$]{Transcendence criteria\\ for
the minimal word of the rational base $3/2$}
\author{Ralf Stephan}
\thanks{Institute for Globally Distributed Open Research and Education (IGDORE)}

\date{\today}

\begin{abstract}
Let $x_{0}$ be a positive integer, let $x_{n}=\lceil 3x_{n-1}/2\rceil$, and let
$w_{n}=2x_{n+1}-3x_{n}\in\{0,1\}$ be the associated word studied by Dubickas; for $x_{0}=1$
the orbit is A061419 and $w$ is the minimal word $g_{3/2}$ of the rational base
number system of Akiyama, Frougny and Sakarovitch. The orbit encodes a real
constant $\Kc=\lim_{n}x_{n}(2/3)^{n}$, equal for $x_{0}=1$ to
$\omega_{3/2}=K(3)=1.6222705028\ldots$, whose irrationality has been open since
1977. We prove that if $w$ is automatic then $\Kc$ is transcendental;
equivalently, an algebraic $\Kc$ forces $w$ to be non-automatic. Further we prove
that either the complexity of $w$ exceeds every
linear bound or $\Kc$ is irrational. All results are
formally verified in the Lean~4 proof assistant \emph{except} four inputs quoted
from the literature, which are assumed as axioms; the transcendence input, a
special case of a theorem of Adamczewski and Faverjon, is among the verified
results rather than among those four.
\end{abstract}

\maketitle

\section{Introduction}\label{sec:intro}

Fix a positive integer $x_{0}$ and iterate the ceiling map
\[
U_{3/2}\colon x\longmapsto\Bigl\lceil\tfrac{3x}{2}\Bigr\rceil,
\qquad x_{n}:=U_{3/2}^{\,n}(x_{0}),
\]
so that $x_{0}<x_{1}<x_{2}<\cdots$. Following Dubickas \cite[p.~245]{Dub09} we
attach to this orbit the word
\[
w_{n}:=2x_{n+1}-3x_{n}\in\{0,1\},\qquad n=0,1,2,\dots,
\]
which is simply the parity of the orbit, $w_{n}=x_{n}\bmod 2$. For $x_{0}=1$ the
orbit is the sequence A061419, and $w$ is an object of the rational base number
system of Akiyama, Frougny and Sakarovitch \cite{AFS08}: their \emph{minimal
word} $g_{3/2}=101100011010011010100110\cdots$ (\cite[\S4.2]{AFS08}, A205083), the
lexicographically least label of an infinite path in the tree $T_{3/2}$. We prove
everything below for a general $x_{0}\ge1$, of
which $g_{3/2}$ is the case $x_{0}=1$.

\begin{remark}\label{rem:shift}
The identification is exact but carries an index shift. Akiyama, Frougny and
Sakarovitch index $g_{3/2}=g_{1}g_{2}g_{3}\cdots$ from $1$ and compute its
letters by $g_{k+1}=qG_{k+1}\bmod p$ \cite[Remark~32]{AFS08}, where $G_{0}=1$,
$G_{k+1}=\lceil (p/q)G_{k}\rceil$ is their Definition~20. At $(p,q)=(3,2)$ the
selection rule $2G_{k+1}=3G_{k}+(G_{k}\bmod 2)$ turns this into
$g_{k+1}=G_{k}\bmod 2$, which is Dubickas's $w_{k}$. So $w_{n}=g_{n+1}$, and the
sequence indexed from $0$ in this paper is the sequence indexed from $1$ there.
\end{remark}

Throughout, $\cx(w,n)$ denotes the \emph{complexity function} (also called the
block-complexity function) of $w$ in the sense of \cite[p.~245]{Dub09}: the
number of distinct vectors $(w_{j},w_{j+1},\dots,w_{j+n-1})$ as $j$ ranges over
all non-negative integers. A sequence is \emph{$k$-automatic} if its $k$-kernel
$\{n\mapsto a(k^{i}n+r):i\ge0,\ 0\le r<k^{i}\}$ is finite, and
\emph{automatic} if it is $k$-automatic for some $k\ge2$ \cite{AS03}. We write
$\nrm{x}$ for the distance from $x$ to the nearest integer.

\subsection{The constant}\label{sub:constant}

Dividing the identity $2x_{n+1}=3x_{n}+w_{n}$ by powers of $3$ (see
Proposition~\ref{prop:partial}) gives $(2/3)^{n}x_{n}=x_{0}+\frac13\sum_{j<n}
w_{j}(2/3)^{j}$, whose limit
\[
\Kc\ :=\ \lim_{n\to\infty}x_{n}(2/3)^{n}\ =\ x_{0}+\tfrac13\sum_{j\ge0}w_{j}
\Bigl(\tfrac23\Bigr)^{j}
\]
exists. For $x_{0}=1$ this is the constant $\omega_{3/2}$ of
\cite[Example~3]{AFS08}, which is the constant $K(3)$ of Odlyzko and Wilf
\cite{OW91}:
\[
\Kc=\omega_{3/2}=K(3)=1.62227\,05028\,84767\,31595\,69509\,8\ldots
\]
(sequence A083286). It is a striking feature of this circle that so little is
known about $\Kc$: its \emph{irrationality} is open, asked by Wang and Washburn
in 1977 \cite{WW77}.

\subsection{Known results}\label{sub:known}

Three facts about $w$ are in the literature. The first two we reprove; the third
we reprove as well (Corollary~\ref{cor:dubfloor}), though we never use it, for
the structural reason recorded in Remark~\ref{rem:linear}.

\begin{theorem}[Odlyzko--Wilf; Akiyama--Frougny--Sakarovitch]\label{thm:cf}
For every $x_{0}\ge1$ and every $n\ge0$,
\[
x_{n}=\bigl\lfloor \Kc\,(3/2)^{n}\bigr\rfloor .
\]
\end{theorem}

This is \cite[Cor.~1]{OW91} in the case $q=3$ of the Josephus problem, where the
relevant map is $x\mapsto\lceil\frac{q}{q-1}x\rceil=\lceil\frac32 x\rceil$, and
independently \cite[Cor.~31]{AFS08}, which gives $G_{k}=\lfloor\gamma_{p/q}
(p/q)^{k+1}\rfloor$ whenever $p\ge 2q-1$ --- at $(3,2)$ one has $3\ge3$. It is
also stated by Dubickas \cite[p.~244]{Dub09}, who records that $x_{n}=\lfloor
c(\beta)\beta^{n}\rfloor$ holds for $\beta\ge2$ or $\beta=2-1/q$, and
$3/2=2-1/2$.

\begin{theorem}[Akiyama--Frougny--Sakarovitch \cite{AFS08}; Dubickas
\cite{Dub09}]\label{thm:aper}
The word $w$ is not eventually periodic.
\end{theorem}

This is \cite[Prop.~26]{AFS08}, and independently \cite[Thm.~2]{Dub09}: an
ultimately periodic word forces the base to be a Pisot or a Salem number, and
$3/2$ is neither --- an inference Dubickas draws himself \cite[p.~245]{Dub09}.

\begin{theorem}[Dubickas \cite{Dub09}]\label{thm:dublin}
$\cx(w,n)>1.70951129\,n$ for every sufficiently large $n$.
\end{theorem}

This is \cite[Cor.~4]{Dub09}, the best lower bound for this word we are
aware of, derived from the general bound $\liminf_{n}\cx(w,n)/n\ge\log
q/\log(p/q)$ of \cite[Thm.~3]{Dub09}; Corollary~\ref{cor:dubfloor} below
recovers it from the rigidity theory by a pigeonhole.

\begin{remark}\label{rem:linear}
Theorem~\ref{thm:dublin} cannot be pushed to non-automaticity by sharpening its
constant. Cobham's theorem \cite{Cob72} bounds the complexity of a $k$-automatic
sequence only by $O(n)$, so a lower bound $\cx(w,n)\ge cn$ is compatible with
automaticity for \emph{every} $c$; the Thue--Morse sequence is $2$-automatic with
$\liminf_{n}\cx(n)/n=3>1.7095$ \cite{Brl89,dLV89}. Every complexity statement
below therefore targets \emph{superlinearity}, not a linear floor.
\end{remark}

\subsection{Results of this paper}\label{sub:new}

The transcendence input of the paper is the Mahler alternative of Adamczewski
and Faverjon, in the following special case.

\begin{theoremintro}[Adamczewski--Faverjon \cite{AF17}, Cor.~1.8; Cobham's 1968
conjecture]\label{thm:af}
Let $(a_{j})_{j\ge0}$ be an automatic sequence of non-negative integers, bounded
by some $B$, and let $\alpha\in\Q$ with $0<|\alpha|<1$. Then either
$\sum_{j\ge0}a_{j}\alpha^{j}$ is transcendental, or it is rational.
\end{theoremintro}

The statement is not ours. It is \cite[Cor.~1.8]{AF17} in the specialization we
need --- $f$ automatic with bounded coefficients, $\alpha$ rational, whence the
field $k$ of the general statement is $\Q$ --- and in that specialization it is
precisely the 1968 conjecture of Cobham, which Adamczewski and Faverjon call the
origin of their paper. What is new here is its status in this paper: it is no
longer assumed. The Lean development derives it, together with each of the three
reductions just named --- automatic $\Rightarrow$ $q$-Mahlerian with a
\emph{nonsingular polynomial} matrix; bounded coefficients $\Rightarrow$ radius
of convergence $\ge1$ $\Rightarrow$ $\alpha$ is not a pole; and $k=\Q$ --- from
two lemmas lying one level below the corollary, Theorems~\ref{thm:af17reg}
and~\ref{thm:af22branch} of \S\ref{sub:axioms}. The route assembles
\cite[\S4]{AF17} with the new proof of the underlying lifting theorem
\cite[\S2]{AF22}, and uses neither the analytic desingularization of
\cite[\S5]{AF17} nor the earlier proof of that theorem.
Appendix~\ref{app:af} records what the exercise turned up: several places where
the printed proof is longer than it needs to be, one printed claim that is false
as stated, one step that cannot be taken in the order in which it is printed,
and a handful of hypotheses that are load-bearing and invisible until one writes
them down.

\begin{remark}[read the alternative literally]\label{rem:branch}
The disjunct ``rational'' is not a degenerate corner. Adamczewski and
Faverjon show it cannot be removed --- \emph{``des exemples montrent que l'on ne
peut se soustraire \`a l'alternative \dots\ m\^eme en supposant la fonction
$f(z)$ transcendante''} --- and their \S8.1 exhibits a $\{0,1\}$-valued
$3$-automatic sequence with $f$ transcendental over $\Q(z)$ and yet
$f_{1}(\varphi)=-\varphi/2\in\Q(\varphi)$. An argument that reads
Theorem~\ref{thm:af} as forcing a contradiction from a \emph{known rational}
value is therefore invalid: the rational value is the permitted branch. This is
why Theorem~\ref{thm:af} is applied below only in the one direction where both
branches die at once (\S\ref{sub:af}).
\end{remark}

Our main result links the arithmetic nature of $\Kc$ to the automaticity of the
word that generates it. To our knowledge no statement of this kind was
previously available.

\begin{theoremintro}\label{thm:main}
Let $x_{0}\ge1$. If the word $w$ is automatic, then $\Kc$ is
transcendental. Equivalently: if $\Kc$ is algebraic, then $w$ is not automatic.
\end{theoremintro}

The reading with content is the first implication: it concludes the \emph{transcendence} of a
constant not known to be \emph{irrational}, from a combinatorial hypothesis on
the word. Theorem~\ref{thm:main} is proved in \S\ref{sec:main} by splitting the
algebraic case in two and closing each half by a different engine --- the
Mahler-method alternative of Adamczewski and Faverjon on the irrational half
(\S\ref{sub:af}), and the Subspace Theorem on the rational half
(\S\ref{sub:kernel}). We stress at once what it does not do.

\begin{remark}[scope]\label{rem:scope}
$\Kc$ is expected to be transcendental, so the hypothesis of the second form of
Theorem~\ref{thm:main} is plausibly vacuous, and the theorem is \emph{not} a
proof that $w$ is non-automatic: the generic, transcendental case is untouched.
\end{remark}

The rational half yields a dichotomy of independent interest, in which both horns
are statements one would like to have and neither is available alone.

\begin{theoremintro}\label{thm:dich}
Let $x_{0}\ge1$. Then either the complexity of $w$ exceeds every linear bound
--- for every $C$ there is an $m\ge1$ with $\cx(w,m)>Cm$, i.e.\
$\limsup_{m}\cx(w,m)/m=\infty$ --- or $\Kc$ is irrational.
\end{theoremintro}

Behind the first horn stands a rigidity theorem with a one-line reformulation
that organizes everything in this paper: length-$m$ factors of $w$ agree
when the orbit values agree modulo $2^{m}$ (Theorem~\ref{thm:rigid}), so
$\cx(w,m)$ \emph{is} the number of residues modulo $2^{m}$ that the orbit
occupies (Proposition~\ref{prop:residues}), and Theorem~\ref{thm:dublin} becomes
a pigeonhole (Corollary~\ref{cor:dubfloor}).

Beyond Theorems~B and~C, \S\ref{sub:algmult} begins the extension of the
Diophantine kernel of \S\ref{sub:kernel} to algebraic multipliers: the
bounded-gap half is proved (on the printed
algebraic-$\delta$ form of the Main Theorem of \cite{CZ04}); the combinatorial
step that turns such a finiteness statement into a superlinear complexity bound
--- pigeonhole, contraction and the growth ceiling, that is, everything in
\S\ref{sub:kernel} except its Diophantine input --- is shown to need no
rationality at all; and the upgrade of
Theorem~\ref{thm:main} to superlinear complexity in every algebraic case is
thereby reduced to a single named statement, a pair theorem over $\Q(\Kc)$.
Unconditionally, close repetitions --- long repeats at bounded distance ---
die out whenever $\Kc$ is algebraic (Corollary~\ref{cor:closerep}), a second
transcendence criterion with a periodicity-adjacent trigger.

Every statement above and below has been formally verified in Lean~4
\emph{except} two classes of statement: the four cited inputs of
\S\ref{sub:axioms}, which are \emph{assumed} and not proved, and the
computations of \S\ref{sub:outlook}, which are not formalized at all.
Everything else is proved in Lean on top of those four axioms.
Appendix~\ref{app:af} describes the formalization of
Theorem~\ref{thm:af} and Appendix~\ref{app:lean} records the Lean name, file and
axiom footprint of each formalization.

\subsection{Cited inputs}\label{sub:axioms}

Four statements are used as black boxes. Each is stated here with a citation and
no proof; on their status as sources see Remark~\ref{rem:sources} below. The
first two are what is left of
Theorem~\ref{thm:af} once it is proved rather than cited: they lie one level
below it, and through it they carry the main results. The third carries the
rational half. The fourth enters only the algebraic-multiplier extension of
\S\ref{sub:algmult}.

\begin{theorem}[Adamczewski--Faverjon \cite{AF17}, Lemme 2.2]\label{thm:af17reg}
Let $q\ge2$, let $A(z)$ be a matrix of polynomials over $\Qbar$ with
$\det A\ne0$, and let $f_{1},\dots,f_{n}$ be power series over $\Qbar$ solving
the $q$-Mahler system $f(z)=A(z)f(z^{q})$. Then the extension
$\Qbar(z)(f_{1},\dots,f_{n})/\Qbar(z)$ is regular --- equivalently, in
characteristic zero, $\Qbar(z)$ is relatively algebraically closed in it.
\end{theorem}

This is \cite[Lemme 2.2]{AF17}. It is what allows a relation whose coefficients
lie in the solution field to be split into relations with coefficients in
$\Qbar(z)$, and it is the only step of the route below that reaches back into the
analytic theory of Mahler functions: its own proof runs through Nishioka's
theorem \cite[Thm.~5.1.7]{Nis96}, in the form ``an algebraic $q$-Mahler function
is rational''. Characteristic zero is not cosmetic --- over $\mathbb{F}_{p}(z)$,
Christol's theorem \cite{Chr79} makes the algebraic power series exactly the
$p$-automatic ones, all of them Mahlerian, and the statement is false.

\begin{theorem}[Adamczewski--Faverjon \cite{AF22}, Lemma 2.8]\label{thm:af22branch}
Let $f_{1},\dots,f_{m}$ solve a $q$-Mahler system and converge on a disc about
$0$, let $\alpha$ be algebraic with $0<|\alpha|<1$, and let $\varphi(z)$ be a
relation matrix for them, invertible over the ambient field of algebraic
functions. Then for all sufficiently large $k$: \emph{(a)} $\alpha^{q^{k}}$ lies
in the disc of convergence of each $f_{i}$; \emph{(b)} each coordinate of
$\varphi(z)$ defines an analytic function on some neighbourhood of
$\alpha^{q^{k}}$; \emph{(c)} the matrix $\varphi(\alpha^{q^{k}})$ is invertible.
\end{theorem}

This is \cite[Lemma 2.8]{AF22}, the one point at which the algebra of the new
proof meets analysis; it is cited because the ingredients of its four-line
literature proof --- a vanishing criterion for the resultant, Newton--Puiseux,
and an implicit function theorem for algebraic function elements --- are absent
from the formalized library. The formalization cites a slightly wider form. The
object Theorem~\ref{thm:af22branch} produces --- a branch of $\varphi$ near
$\alpha^{q^{k}}$, and with it a realization of the algebraic functions in play by
honest analytic ones --- is consumed five times, and the four further properties
consumed of it (the solutions realized alongside $\varphi$, which is clause
\emph{(a)}; that realization analytic and injective; the neighbourhood connected;
and the corresponding branch of the twisted matrix $\varphi(z^{q^{k}})$ at
$\alpha$, which takes the same values) all hold of that one object by
construction. Folding them into the citation keeps the count of cited axioms at
two rather than five; Appendix~\ref{app:af} says why none of them is a further
assumption.

\begin{remark}[on the sources]\label{rem:sources}
All four are refereed. Theorem~\ref{thm:af22branch} moreover asks the least
trust of the four, since the statement itself is classical and is not what
\cite{AF22} is for: an element algebraic over $\Qbar(z)$ has an analytic
branch at every point outside a finite set, and a matrix with nonzero
determinant is singular only at finitely many points; since the $\alpha^{q^{k}}$
are pairwise distinct and tend to $0$, \emph{(b)} and \emph{(c)} therefore fail
for at most finitely many $k$, and \emph{(a)} holds for all large $k$. What is
not classical --- and what is done here
rather than cited --- is everything \cite{AF22} builds on top of it.
\end{remark}

\begin{theorem}[Subspace Theorem; Schmidt \cite{Sch91}, Evertse--Schlickewei
form \cite{BG06}]\label{thm:subspace}
Let $K$ be a number field, $n\ge2$, and $S$ a finite set of places of $K$
containing the archimedean ones. For each $v\in S$ let $L_{v,1},\dots,L_{v,n}$
be linearly independent linear forms in $X_{1},\dots,X_{n}$ with coefficients in
$K$. For every $\varepsilon>0$, the nonzero $x\in K^{n}$ with
\[
\prod_{v\in S}\prod_{i=1}^{n}\frac{|L_{v,i}(x)|_{v}}{\nrm{x}_{v}}
\ \le\ H(x)^{-n-\varepsilon}
\]
lie in finitely many proper linear subspaces of $K^{n}$.
\end{theorem}

This is \cite[Thm.~1D$'$]{Sch91}. It is the only Diophantine input of
\S\ref{sub:kernel}: both of the approximation theorems used there --- the Main
Theorem of Corvaja and Zannier \cite{CZ04} and the modified pair theorem of Nair, Kumar
and Rout \cite{NKR25} --- are \emph{derived} from it in the formalization rather
than assumed, at $n=2$ and $n=3$ respectively.

\begin{remark}[on \cite{NKR25}]\label{rem:nkr}
Nothing is taken on the authority of \cite{NKR25}, an unrefereed preprint whose
Theorem~1.3(i) is false as printed. Specialized to $\Q$ and to pairs, which is
how we use it, it asserts: if $\alpha_{1},\alpha_{2}\in\Q^{*}$,
$\varepsilon_{1}>0$, and an infinite family of pairs $(u_{1},u_{2})$ drawn from
a finitely generated $\Gamma\le\Q^{*}$ satisfies $|u_{i}|\ge1$,
$u_{1}\ne-u_{2}$, pairwise-distinct ratios in both orientations, and
\[
\nrm{\alpha_{1}u_{1}+\alpha_{2}u_{2}}<\bigl(H(u_{1})H(u_{2})\bigr)^{-\varepsilon_{1}},
\]
then some member of the family has both entries in $\Z$. What is missing is the
strict positivity $\nrm{\alpha_{1}u_{1}+\alpha_{2}u_{2}}>0$, which their own
Theorem~1.1(iv) does carry. Without it, take $\alpha_{1}=\alpha_{2}=1$,
$\Gamma=\langle2,3\rangle$ and $(u_{1},u_{2})=(3^{n}/2,\,3^{2n}/2)$ for $n\ge1$:
the sum is an integer by parity, so the left-hand side is $0$ and every printed
hypothesis holds, while no entry is ever an integer. That refutation is
machine-checked in the formalization. The step of their \S4.1 proof that fails is the
passage to a uniform $\varepsilon$ --- their $\kappa$, and with it
$\varepsilon$, depends on the tuple, while their Lemma~2.2 wants one
$\varepsilon$ for the whole family.

Reinstating $\nrm{\alpha_{1}u_{1}+\alpha_{2}u_{2}}>0$ as a hypothesis repairs
the statement, and that repaired form is the one used here --- derived over $\Q$
from Theorem~\ref{thm:subspace} at $n=3$ rather than cited, with the positivity
discharged by parity where it is needed (Lemma~\ref{lem:degen}). We are still in
debt of the NKR preprint as the source of an important argument.
\end{remark}

\begin{theorem}[Corvaja--Zannier \cite{CZ04}, Main Theorem]\label{thm:czalg}
Let $\Gamma\subset\overline{\Q}{}^{\,*}$ be a finitely generated multiplicative
group, let $\delta$ be a nonzero algebraic number, and let $\varepsilon>0$.
Then there are only finitely many pairs $(q,u)\in\Z\times\Gamma$,
$d:=[\Q(u):\Q]$, such that $|\delta qu|>1$, $\delta qu$ is not pseudo-Pisot,
and
\[
0\ <\ \nrm{\delta qu}\ <\ H(u)^{-\varepsilon}\,q^{-d-\varepsilon}.
\]
\end{theorem}

Here a real algebraic $\alpha$ is \emph{pseudo-Pisot} if $|\alpha|>1$, every
conjugate of $\alpha$ other than $\alpha$ itself has modulus $<1$, and the
trace of $\alpha$ is a rational integer; a pseudo-Pisot number need not be an
algebraic integer. For $\delta\in\Q$ this theorem is \emph{derived} from
Theorem~\ref{thm:subspace} in the formalization. For
algebraic irrational $\delta$ it is taken on \cite{CZ04}'s authority as a
cited axiom, and it is used only in \S\ref{sub:algmult}: a \emph{lane} of its
own --- an axiom footprint disjoint from those of the other three, carried by no
result outside that subsection.

\section{Introduction and proof overview}\label{sec:overview}

\subsection*{One object, two descriptions}

Start at $1$ and repeatedly multiply by $3/2$, rounding up:
\[
1,\ 2,\ 3,\ 5,\ 8,\ 12,\ 18,\ 27,\ 41,\ 62,\ 93,\ 140,\ 210,\ 315,\ 473,\ \dots
\]
Multiplying by $3/2$ lands on an integer exactly when the number is even; when
it is odd the ceiling adds a half. Record a $1$ where it does and a $0$ where it
does not, and the record is the word
\[
w\ =\ 1\,0\,1\,1\,0\,0\,0\,1\,1\,0\,1\,0\,0\,1\,1\,0\cdots
\]
--- which is simply the parity of the orbit, since the rounding bites exactly at
the odd terms.

The word is not a lossy summary of the orbit; it is the orbit. The terms grow
like $(3/2)^{n}$ at a definite rate
\[
\Kc\ =\ 1.6222705\ldots,
\]
and the whole orbit can be read back off that single real number:
$x_{n}=\lfloor\Kc(3/2)^{n}\rfloor$ (Theorem~\ref{thm:cf}). So $w$ and $\Kc$ are
two descriptions of one object --- $w$ is the digit string of $\Kc$ in the
rational base $3/2$, and $\Kc$ is the number that string names. Two questions
about them have been open for decades: is $\Kc$ irrational (asked in 1977), and
is $w$ automatic, that is, produced by a finite automaton reading the digits of
$n$? Neither is answered here. What is proved is that they cannot both go the
easy way: \emph{if $w$ is automatic then $\Kc$ is transcendental}
(Theorem~\ref{thm:main}).

\subsection*{The word is an odometer reading}

Everything rests on one dictionary. The $m$ letters of $w$ read from position
$n$ depend on nothing but $x_{n}\bmod2^{m}$, and they determine it
(Theorem~\ref{thm:rigid}). The word is the low-order end of a counter: reading a
block of $m$ letters is reading $m$ bits of $x_{n}$, and two positions show the
same block exactly when the orbit values agree modulo $2^{m}$.

For instance $x_{1}=2$ and $x_{6}=18$ differ by $16=2^{4}$, and the four letters
from position $1$ and from position $6$ do coincide: both read $0110$. Counting
blocks is therefore counting residues, and the complexity $\cx(w,m)$ --- the
number of distinct length-$m$ blocks --- \emph{equals} the number of residues
modulo $2^{m}$ that the orbit ever occupies
(Proposition~\ref{prop:residues}).

\subsection*{Repetitions are paid for in powers of two}

That example also shows what limits the word. A repeated block of length $k$ at
positions $a<c$ forces $2^{k}\mid x_{c}-x_{a}$; but by step $c$ the orbit has
only reached about $(3/2)^{c}$, so the divisibility cannot run deep:
\[
k\ \le\ \tfrac{24}{41}\,c+1\ \approx\ 0.585\,c
\qquad\text{(Corollary~\ref{cor:slope}).}
\]
A repetition is a coincidence that must be paid for in powers of $2$, and by
time $c$ the orbit has earned only $c\log_{2}(3/2)$ of them. This one inequality
does three jobs of increasing depth. It makes $w$ aperiodic, since a period
would supply repetitions of \emph{every} length at one fixed $c$
(Theorem~\ref{thm:aper2}). It makes the known linear lower bound
$\cx(w,n)>1.7095\,n$ a pigeonhole: the first $1.7095\,m$ orbit points are simply
too small to collide modulo $2^{m}$ (Corollary~\ref{cor:dubfloor}). And it
closes off a whole family of attacks --- periodize the digits at a repetition
and feed the resulting approximants to the $p$-adic Subspace Theorem, as
\cite{AB07} did for Loxton--van der Poorten --- because that route needs
repetitions of relative length above $1.7095$ where the ceiling allows $0.585$
(\S\ref{sub:ledger}). The two numbers are reciprocals straddling $1$, so no
sharpening closes the gap.

\subsection*{Two engines for two halves}

Suppose $\Kc$ is algebraic and $w$ is automatic, and derive a contradiction.
$\Kc$ is either irrational or rational, and the halves are closed by arguments
with nothing in common.

\emph{The irrational half} (\S\ref{sub:af}) is Mahler's method. An automatic
sequence has a generating series satisfying $f(z)=A(z)f(z^{q})$ with $A$ a
matrix of polynomials, and for such a series Adamczewski and Faverjon prove an
alternative: at a nonzero rational point of the unit disc its value is
transcendental or rational (Theorem~\ref{thm:af}). Here that value is
$3(\Kc-x_{0})$, so the alternative transfers to $\Kc$ verbatim, and the two
exits are shut by the two hypotheses --- ``transcendental'' contradicts the
algebraicity, ``rational'' contradicts the irrationality. It is a trap with two
doors, and each assumption closes one.

\emph{The rational half} (\S\ref{sub:kernel}) is Diophantine approximation. If
$\Kc=\delta$ is rational, a repeated block of length $k$ at $(a,c)$ pushes the
number $\delta((3/2)^{c}-(3/2)^{a})$ to within $(2/3)^{k}$ of an integer
(Proposition~\ref{prop:violator}): a repetition in the word \emph{is} a rational
number caught unusually close to an integer. Such coincidences are finite in
number --- that is the kernel, Theorem~\ref{thm:kernel}, and it is where the
Subspace Theorem is spent. A linear bound $\cx(w,m)\le Cm$, on the other hand,
would manufacture one at every scale by pigeonhole. So no linear bound holds
(Theorem~\ref{thm:t1arat}), and Cobham's theorem forbids that of an automatic
word.

The simplest case of the kernel is $\delta=1$: how close can
$(3/2)^{c}-(3/2)^{a}$ come to an integer? That case is \cite{RS26}. Allowing an
arbitrary rational multiplier is not free, and the step that breaks outright is
the most innocent-looking one: for $\delta\ne1$ the value can be an integer
exactly --- at $\delta=4$, $a=1$, $c=2$ it is $3$ --- and about an integer no
approximation theorem has anything to say. The damage is confined to a finite
box and discarded (Lemma~\ref{lem:degen}).

Neither engine can do the other's work. The Mahler alternative sees the word
only through the single real number it defines, so it yields nothing
quantitative about complexity; the Diophantine kernel sees only how well that
number is approximated, and needs it rational to begin with.
\S\ref{sub:algmult} narrows the gap from the second side, weakening
``rational'' towards ``algebraic'' as far as the printed literature allows.

\subsection*{What is assumed, and where}

Four statements are used without proof (\S\ref{sub:axioms}), and they do not
mix. Two lemmas from the Adamczewski--Faverjon circle power the irrational half
and only it; the Subspace Theorem powers the rational half and only it; the
Corvaja--Zannier theorem at algebraic $\delta$ is spent entirely inside
\S\ref{sub:algmult}. Everything else --- all of \S\ref{sec:basic},
\S\ref{sec:rig} and \S\ref{sub:ledger} included --- is proved outright and
machine-checked. Figure~\ref{fig:dep} shows the shape of the dependency,
Table~\ref{tab:dep} the exact footprint of each statement.

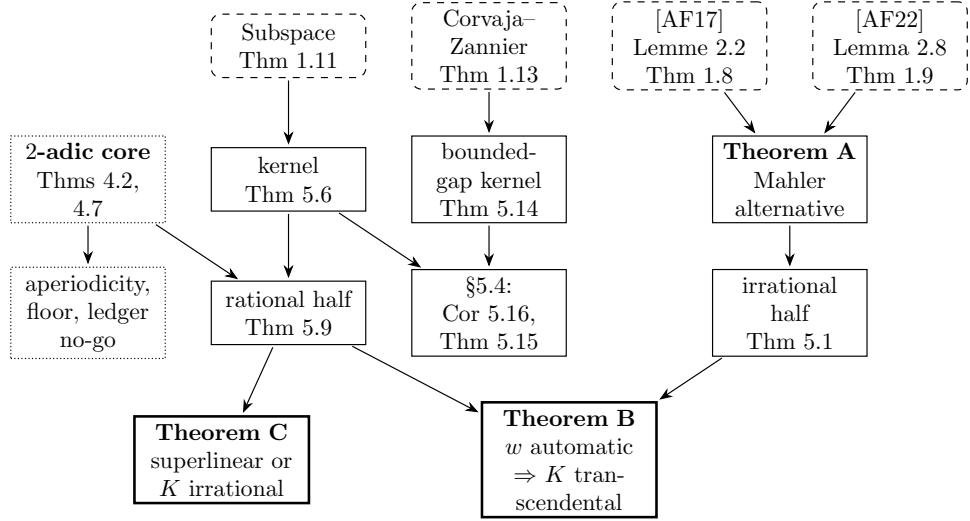
\begin{figure}[htb]
\centering
\fitwidth{%
\begin{tikzpicture}[>={Stealth[length=2mm]}, line width=0.5pt,
  ax/.style={draw,dashed,rounded corners,align=center,font=\small,
             text width=2.1cm,inner sep=3pt,minimum height=0.95cm},
  res/.style={draw,align=center,font=\small,
             text width=2.1cm,inner sep=3pt,minimum height=0.95cm},
  core/.style={draw,densely dotted,align=center,font=\small,
             text width=2.1cm,inner sep=3pt,minimum height=0.95cm},
  key/.style={draw,line width=1.1pt,align=center,font=\small,
             text width=2.4cm,inner sep=3pt,minimum height=0.95cm},
  ar/.style={->,shorten >=1pt,shorten <=1pt}]

\node[ax] (asub)  at (0,0)   {Subspace\\ Thm~\ref{thm:subspace}};
\node[ax] (acz)   at (3,0)   {Corvaja--Zannier\\ Thm~\ref{thm:czalg}};
\node[ax] (aaf17) at (6,0)   {[AF17] Lemme 2.2\\ Thm~\ref{thm:af17reg}};
\node[ax] (aaf22) at (9,0)   {[AF22] Lemma 2.8\\ Thm~\ref{thm:af22branch}};

\node[core] (core) at (-3,-2)   {\textbf{$2$-adic core}\\ Thms~\ref{thm:rigid},
                                 \ref{thm:ceiling}};
\node[res]  (ker)  at (0,-2)    {kernel\\ Thm~\ref{thm:kernel}};
\node[res]  (sl)   at (3,-2)    {bounded-gap kernel\\ Thm~\ref{thm:algslice}};
\node[res]  (ta)   at (7.5,-2)  {\textbf{Theorem A}\\ Mahler alternative};

\node[core] (free) at (-3,-4)   {aperiodicity, floor, ledger no-go};
\node[res]  (rat)  at (0,-4)    {rational half\\ Thm~\ref{thm:t1arat}};
\node[res]  (alg)  at (3,-4)    {\S\ref{sub:algmult}: Cor~\ref{cor:closerep},
                                 Thm~\ref{thm:algcond}};
\node[res]  (irr)  at (7.5,-4)  {irrational half\\ Thm~\ref{thm:t1a}};

\node[key] (C) at (-1,-6.2) {\textbf{Theorem C}\\ superlinear or $\Kc$
                             irrational};
\node[key] (B) at (4.2,-6.2) {\textbf{Theorem B}\\ $w$ automatic
                              $\Rightarrow$ $\Kc$ transcendental};

\draw[ar] (asub)  -- (ker);
\draw[ar] (acz)   -- (sl);
\draw[ar] (aaf17) -- (ta);
\draw[ar] (aaf22) -- (ta);
\draw[ar] (core)  -- (free);
\draw[ar] (core)  -- (rat);
\draw[ar] (ker)   -- (rat);
\draw[ar] (ker)   -- (alg);
\draw[ar] (sl)    -- (alg);
\draw[ar] (ta)    -- (irr);
\draw[ar] (rat)   -- (C);
\draw[ar] (rat)   -- (B);
\draw[ar] (irr)   -- (B);
\end{tikzpicture}}
\caption{Logical dependencies. Dashed: the four cited inputs of
\S\ref{sub:axioms}. Dotted: the $2$-adic core of \S\ref{sec:rig}, which uses
none of them and carries aperiodicity (Theorem~\ref{thm:aper2}), the complexity
floor (Corollary~\ref{cor:dubfloor}) and the no-go of \S\ref{sub:ledger}
(Theorem~\ref{thm:noindex}) on its own. An arrow reads ``is used in the proof
of''. The two lanes never meet: no result carries both an \textnormal{[AF]}
lemma and the Corvaja--Zannier axiom.}
\label{fig:dep}
\end{figure}

\begin{table}[htb]
\centering
\small
\renewcommand{\arraystretch}{1.15}
\begin{tabular}{@{}llcccc@{}}
\hline
Statement & \S & [AF17] & [AF22] & [Sub] & [CZ04a]\\
\hline
Thm~\ref{thm:rigid} (rigidity) & \ref{sec:rig} & & & & \\
Prop~\ref{prop:residues} (complexity $=$ residues) & \ref{sec:rig} & & & & \\
Cor~\ref{cor:dubfloor} (complexity floor) & \ref{sec:rig} & & & & \\
Thm~\ref{thm:ceiling}, Cor~\ref{cor:slope} (ceiling) & \ref{sec:rig} & & & & \\
Thm~\ref{thm:aper2} (aperiodicity) & \ref{sec:rig} & & & & \\
Thm~\ref{thm:cf} (closed form) & \ref{sec:rig} & & & & \\
Thm~\ref{thm:af} (\textbf{Theorem A}) & \ref{sub:new} & \checkmark & \checkmark
  & & \\
Thm~\ref{thm:t1a}, Cor~\ref{cor:t1acontra} (irrational half) & \ref{sub:af}
  & \checkmark & \checkmark & & \\
Prop~\ref{prop:violator}, Cor~\ref{cor:contract} (contraction)
  & \ref{sub:kernel} & & & & \\
Thm~\ref{thm:kernel} (kernel) & \ref{sub:kernel} & & & \checkmark & \\
Cor~\ref{cor:effslice} (effective slice) & \ref{sub:kernel} & & & & \\
Thm~\ref{thm:t1arat} (rational half) & \ref{sub:kernel} & & & \checkmark & \\
Thm~\ref{thm:dich} (\textbf{Theorem C}) & \ref{sub:kernel} & & & \checkmark & \\
Thm~\ref{thm:main} (\textbf{Theorem B}) & \ref{sub:compose} & \checkmark
  & \checkmark & \checkmark & \\
Thm~\ref{thm:algslice} (bounded-gap kernel) & \ref{sub:algmult} & & & &
  \checkmark\\
Thm~\ref{thm:algcond}, Cor~\ref{cor:closerep} & \ref{sub:algmult} & & &
  \checkmark & \checkmark\\
Thm~\ref{thm:noindex} (ledger no-go) & \ref{sub:ledger} & & & & \\
Prop~\ref{prop:padic} ($2$-adic value) & \ref{sub:padic} & & & & \\
\hline
\end{tabular}
\caption{Which cited input each statement carries; a blank row is proved from
nothing but classical logic. \textnormal{[AF17]} and \textnormal{[AF22]} are
Theorems~\ref{thm:af17reg} and~\ref{thm:af22branch}, \textnormal{[Sub]} is
Theorem~\ref{thm:subspace}, \textnormal{[CZ04a]} is Theorem~\ref{thm:czalg} at
algebraic $\delta$. Appendix~\ref{app:lean} gives the same information per Lean
declaration, as confirmed by \texttt{\#print axioms}.}
\label{tab:dep}
\end{table}

\subsection*{The route through the paper}

\S\ref{sec:basic} fixes the orbit, the word and the constant. \S\ref{sec:rig}
proves the dictionary and the ceiling, and with them aperiodicity and the closed
form. \S\ref{sec:main} runs the two engines and assembles
Theorem~\ref{thm:main}, then pushes the Diophantine one towards algebraic
multipliers. \S\ref{sec:comp} collects two complements --- the $2$-adic value of
the word, and the ledger that closes the Subspace route --- and reports the
computations. Appendix~\ref{app:af} records what formalizing the Mahler input
turned up about the printed proofs, and Appendix~\ref{app:lean} indexes every
statement by its Lean name.

\section{Basics}\label{sec:basic}

\begin{definition}\label{def:orbit}
Given $x_{0}\in\N$, set $x_{n+1}:=\lceil 3x_{n}/2\rceil$ for $n\ge0$, and let
$w_{n}:=x_{n}\bmod 2$.
\end{definition}

\begin{proposition}\label{prop:sel}
For every $n$, $\;2x_{n+1}=3x_{n}+w_{n}$, and $w_{n}\in\{0,1\}$. If $x_{0}\ge1$
then $x_{0}<x_{1}<x_{2}<\cdots$.
\end{proposition}

The identity $2x_{n+1}=3x_{n}+w_{n}$ is the whole engine: $w_{n}$ is the
correction that makes $3x_{n}$ even, so $w_{n}=2x_{n+1}-3x_{n}$ is Dubickas's
word and $w_{n}\equiv x_{n}\pmod 2$ simultaneously.

\begin{definition}\label{def:W}
For $a,k\in\N$ put $\;W(a,k):=\sum_{i<k}3^{\,k-1-i}2^{i}w_{a+i}$, the
\emph{circuit sum} of the length-$k$ factor of $w$ at $a$.
\end{definition}

\begin{proposition}\label{prop:circuit}
For all $a,k\in\N$, $\;2^{k}x_{a+k}=3^{k}x_{a}+W(a,k)$.
\end{proposition}

\begin{proof}[Proof sketch]
Induction on $k$, feeding $2x_{a+k+1}=3x_{a+k}+w_{a+k}$ into the inductive
hypothesis and collecting terms; the recurrence $W(a,k+1)=3W(a,k)+2^{k}w_{a+k}$
matches.
\end{proof}

\begin{proposition}\label{prop:partial}
For every $n$,
\[
x_{n}\Bigl(\tfrac23\Bigr)^{n}=x_{0}+\tfrac13\sum_{j<n}w_{j}
\Bigl(\tfrac23\Bigr)^{j}.
\]
\end{proposition}

\begin{proof}[Proof sketch]
This is Proposition~\ref{prop:circuit} at $a=0$ divided by $3^{n}$; the
formalization proves it by direct induction instead, which keeps the natural
subtraction implicit in $W$ out of $\R$.
\end{proof}

Since $w_{j}\in\{0,1\}$, the series $\sum_{j}w_{j}(2/3)^{j}$ is dominated by a
geometric series and converges; all sums below are finite for this reason.

\begin{definition}\label{def:K}
$\displaystyle \Kc:=x_{0}+\tfrac13\sum_{j\ge0}w_{j}(2/3)^{j}$.
\end{definition}

\begin{proposition}\label{prop:Klim}
$\Kc=\lim_{n\to\infty}x_{n}(2/3)^{n}$, and
\[
\sum_{j\ge0}w_{j}\Bigl(\tfrac23\Bigr)^{j}=3(\Kc-x_{0}).
\]
\end{proposition}

Defining $\Kc$ by the series rather than by the limit makes the second identity
free and the limit description a theorem; the second identity is the interface
through which \S\ref{sub:af} feeds the word to Theorem~\ref{thm:af}, and is the
reason for the choice.

\begin{proposition}\label{prop:window}
For every $n$, $\;x_{n}\le \Kc\,(3/2)^{n}\le x_{n}+1$.
\end{proposition}

\begin{proof}[Proof sketch]
Split the series of Definition~\ref{def:K} at $n$. The head is
$x_{n}(2/3)^{n}$ by Proposition~\ref{prop:partial}; the tail is non-negative,
giving the left inequality, and is bounded by $\sum_{i\ge0}(2/3)^{n+i}
=3(2/3)^{n}$, giving the right one after multiplying by $(3/2)^{n}$.
\end{proof}

Proposition~\ref{prop:window} pins the orbit to a half-open unit window below
$\Kc(3/2)^{n}$; sharpening its right-hand inequality to a strict one leads immediaely
to Theorem~\ref{thm:cf}, and costs the aperiodicity of $w$. That is the
subject of the next section.

\section{Two-adic rigidity and aperiodicity of \texorpdfstring{$w$}{w}}\label{sec:rig}

\begin{definition}\label{def:rep}
A \emph{repeated factor} of length $k$ at $(a,c)$ is an equality of the length-$k$
factors of $w$ at $a$ and at $c$: $w_{a+i}=w_{c+i}$ for all $i<k$. Occurrences
may overlap.
\end{definition}

\begin{theorem}\label{thm:rigid}
For all $a,b,m$, the length-$m$ factors of $w$ at $a$ and $b$ agree if and only
if $2^{m}\mid x_{b}-x_{a}$.
\end{theorem}

\begin{proof}[Proof]
$(\Rightarrow)$ Equal factors give equal circuit sums, $W(a,m)=W(b,m)$;
subtracting the two instances of Proposition~\ref{prop:circuit} leaves
\[
    3^{m}(x_{b}-x_{a})=2^{m}(x_{b+m}-x_{a+m})\text{, so }\ 2^{m}\mid 3^{m}(x_{b}-x_{a}),
\]
and $2\nmid3$.

$(\Leftarrow)$ is a $2$-adic descent: from $2^{m}\mid x_{a}-x_{b}$,
one power of $2$ survives at each step, so $w_{a+i}=w_{b+i}$ (the letters are
the residues mod $2$), and then $2(x_{a+i+1}-x_{b+i+1})=3(x_{a+i}-x_{b+i})$
divides the agreement by one power of $2$. The agreement therefore
decays to $2^{m-i}\mid x_{a+i}-x_{b+i}$ and stays alive $m$ steps ---
the length of the factor.
\end{proof}

So the length-$m$ factors of $w$ \emph{are} the residues $x_{n}\bmod 2^{m}$; in
particular, counting factors is counting residues.

\begin{proposition}\label{prop:residues}
For every $x_{0}\ge1$ and every $m$,
\[
\cx(w,m)\;=\;\#\bigl\{\,x_{n}\bmod 2^{m}\ :\ n\ge0\,\bigr\}.
\]
\end{proposition}

\begin{proof}[Proof sketch]
Immediate from Theorem~\ref{thm:rigid}: two positions carry the same length-$m$
factor exactly when they carry the same residue, so factors and occupied
residues are in bijection.
\end{proof}

The restatement moves the complexity of $w$ into the $2$-adic dynamics of
$U_{3/2}$, and it has consequences the combinatorial formulation hides. The
first is that the linear lower bound of Theorem~\ref{thm:dublin} is a
pigeonhole.

\begin{corollary}\label{cor:dubfloor}
Let $x_{0}=1$. For every $m\ge1$,
\[
\cx(w,m)\;\ge\;\Bigl\lfloor\tfrac{41(m-1)}{24}\Bigr\rfloor+1 .
\]
\end{corollary}

\begin{proof}[Proof sketch]
The orbit values below $2^{m}$ are strictly increasing, hence pairwise distinct,
hence --- being below the modulus --- pairwise distinct residues modulo
$2^{m}$, and Proposition~\ref{prop:residues} counts each as a separate factor.
By the growth bound inside Theorem~\ref{thm:ceiling},
$x_{j}<(x_{0}+1)(3/2)^{j}$, the orbit stays below $2^{m}$ for more than
$m\log2/\log(3/2)-O(1)$ steps; the certificate $3^{41}\le2^{65}$ turns the count
into the stated integer form, of slope $41/24=1.7083\ldots$
\end{proof}

\begin{remark}\label{rem:dubfloor}
Run with real logarithms instead of the integer certificate, the same count
gives $\cx(w,m)\ge m\log2/\log(3/2)-O_{x_{0}}(1)$ with
$\log2/\log(3/2)=1.70951\ldots$, which is Theorem~\ref{thm:dublin} up to its
additive constant --- and, read through Proposition~\ref{prop:residues},
explains it: the first $1.7095\,m$ orbit points are simply too small to collide
modulo $2^{m}$. Improving on the constant therefore requires orbit points
\emph{beyond} that initial segment to occupy new residue classes, which is
exactly what the $2$-adic behaviour of the orbit controls.
\S\ref{sub:outlook} reports computations: empirically the orbit occupies
every residue class modulo $2^{m}$, at the saturation rate of a uniformly
random sequence of residues.
\end{remark}

The reachable-residue question behind Remark~\ref{rem:dubfloor} makes sense at
every modulus, not only at the powers of $2$ that the factors of $w$ see; the first
instance beyond those is free.

\begin{proposition}\label{prop:mod3}
For every $x_{0}$ and every $n$, $\;x_{n+1}\not\equiv1\pmod3$.
\end{proposition}

\begin{proof}[Proof sketch]
$2x_{n+1}=3x_{n}+w_{n}$ gives $x_{n+1}\equiv2w_{n}\pmod3$, and $w_{n}\in\{0,1\}$.
\end{proof}

The dictionary of Proposition~\ref{prop:residues} is also a cost, not only a
convenience: a repeated factor forces a large power of $2$ to divide a
difference of orbit values, while the orbit has only $(3/2)^{c}$ room.

\begin{theorem}\label{thm:ceiling}
Let $x_{0}\ge1$ and $a<c$. If $w$ has a repeated factor of length $k$ at $(a,c)$,
then
\[
2^{\,k+c}<3^{c}(x_{0}+1).
\]
\end{theorem}

\begin{proof}[Proof sketch]
By Theorem~\ref{thm:rigid}, $2^{k}\mid x_{c}-x_{a}$, and $x_{a}<x_{c}$ by
Proposition~\ref{prop:sel}, so $2^{k}\le x_{c}-x_{a}<x_{c}+1$; multiplying by
$2^{c}$ gives $2^{k+c}<2^{c}(x_{c}+1)$. It remains to see that the orbit has only
$(3/2)^{c}$ room, i.e.\ that $2^{n}(x_{n}+1)\le3^{n}(x_{0}+1)$ for every $n$: the
shifted quantity $x_{n}+1$ is submultiplicative under the ceiling rule, since
$2(x_{n+1}+1)=3x_{n}+w_{n}+2\le3(x_{n}+1)$ because $w_{n}\le1$, so induction
applies. At $n=c$ this chains onto the display.
\end{proof}

\begin{corollary}\label{cor:slope}
Let $x_{0}=1$ and $a<c$. A repeated factor of length $k$ at $(a,c)$ obeys the integer
inequality $41k\le24c+40$; that is, $k\le\frac{24}{41}c+1\approx0.585\,c$.
\end{corollary}

\begin{proof}[Proof sketch]
Raise Theorem~\ref{thm:ceiling} to the $41$st power and apply the integer
certificate $3^{41}\le2^{65}$, which encodes $\log_{2}3\le65/41$ without reals
or logarithms. The resulting exponent inequality is $41(k+c)<65c+41$.
\end{proof}

The slope $24/41>\log_{2}(3/2)=0.5849\ldots$ is the certified form of the
repeated-factor ceiling. The ceiling itself is exact and loses nothing: in real
logarithms Theorem~\ref{thm:ceiling} reads
$k<c\log_{2}(3/2)+\log_{2}(x_{0}+1)$, with slope exactly $\log_{2}(3/2)$;
Corollary~\ref{cor:slope} is its integer-certified approximation, which is what
the formal statements downstream consume. \S\ref{sub:ledger} shows that this
ceiling is what closes off a whole family of attacks on the word.

\begin{theorem}[= Theorem~\ref{thm:aper}]\label{thm:aper2}
For $x_{0}\ge1$, the word $w$ is not eventually periodic: there are no $N$ and
$p\ge1$ with $w_{n+p}=w_{n}$ for all $n\ge N$.
\end{theorem}

\begin{proof}[Proof sketch]
Eventual periodicity gives a repeated factor at the \emph{fixed} pair $(N,N+p)$ of
\emph{every} length $k$. But Theorem~\ref{thm:ceiling} caps every such $k$ by
the single number $3^{c}(x_{0}+1)$ with $c=N+p$ fixed; taking
$k=3^{c}(x_{0}+1)$ and using $k<2^{k}\le2^{k+c}$ contradicts it.
\end{proof}

\begin{proof}[Proof of Theorem~\ref{thm:cf}]
Write $\Kc(y)$ for the constant of the orbit started at $y$. The orbit is
shift-invariant --- running $n+j$ steps from $x_{0}$ is running $j$ steps from
$x_{n}$ --- so the word is too, and splitting the series of
Definition~\ref{def:K} at $n$ and reindexing gives $\Kc(x_{n})=\Kc(x_{0})
(3/2)^{n}$. It therefore suffices to treat $n=0$ at the initial condition
$x_{n}$, which is positive: that is, to show $x_{n}\le \Kc(x_{n})<x_{n}+1$.

The left inequality is Proposition~\ref{prop:window}. For the right one it is
enough to see that $\Kc<x_{0}+1$ strictly for every $x_{0}\ge1$, and then to
apply this at $x_{n}$. Some letter of $w$ is $0$: an all-ones word would be
periodic with period $1$, contradicting Theorem~\ref{thm:aper2}. So
$\sum_{j}w_{j}(2/3)^{j}<\sum_{j}(2/3)^{j}=3$, the terms being dominated termwise
and one of them strictly, and Definition~\ref{def:K} turns this into
$\Kc<x_{0}+1$.
\end{proof}

\section{Transcendence criteria for \texorpdfstring{$\Kc$}{K}}\label{sec:main}

We prove Theorem~\ref{thm:main}. Suppose $\Kc$ is algebraic. Either it is
irrational, or it is rational. The two halves are closed by unrelated engines,
and \S\ref{sub:compose} assembles them.

\subsection{The algebraic irrational half}\label{sub:af}

\begin{theorem}\label{thm:t1a}
Let $x_{0}\ge1$. If $\Kc$ is an algebraic irrational, then $w$ is not automatic.
\end{theorem}

\begin{proof}
Suppose $w$ is automatic. Its letters lie in $\{0,1\}$
(Proposition~\ref{prop:sel}), so they are bounded, and $\alpha=2/3$ is rational
with $0<|\alpha|<1$. Theorem~\ref{thm:af} therefore applies and gives: either
$\sum_{j}w_{j}(2/3)^{j}$ is transcendental, or it is rational. By
Proposition~\ref{prop:Klim} that sum is $3(\Kc-x_{0})$ with $x_{0}\in\N$, so the
alternative transfers verbatim to $\Kc$: the first branch would make $\Kc$
transcendental, contradicting algebraicity, and the second would make
$\Kc=r/3+x_{0}$ rational for some $r\in\Q$, contradicting irrationality. Both
branches die, so $w$ is not automatic.
\end{proof}

\begin{corollary}\label{cor:t1acontra}
If $w$ is automatic and $\Kc$ is irrational, then $\Kc$ is transcendental.
\end{corollary}

\subsection{The rational half: a Diophantine kernel}\label{sub:kernel}

Assume now $\Kc=\delta\in\Q$. We show that no linear function bounds $\cx(w,\cdot)$
from above, from which non-automaticity follows by Cobham's theorem. The route is the one used in
\cite{RS26} for the steering word of $(3/2)^{n}$, with the multiplier slot opened
up; we indicate below the three places where $\delta\ne1$ costs a real argument.

Write $\theta_{n}:=\Kc(3/2)^{n}-x_{n}$ for the \emph{deviation} of the orbit. By
Theorem~\ref{thm:cf}, $\theta_{n}\in[0,1)$, also $\theta_{n}$ is
the fractional part $\{\Kc(3/2)^{n}\}$. The deviations are the
Vijayaraghavan sequence of $\Kc$ at ratio $3/2$, the object of the
$(3/2)^{n}$-distribution literature (Flatto--Lagarias--Pollington \cite{FLP95}),
and every statement about $\theta$ below is a statement about
$\{\Kc(3/2)^{n}\}$.

\begin{proposition}\label{prop:violator}
Let $x_{0}\ge1$ and suppose $\Kc=\delta\in\Q$. A repeated factor of length $k$ at
$(a,c)$ forces
\[
\bigl\lVert\,\delta\bigl((3/2)^{c}-(3/2)^{a}\bigr)\,\bigr\rVert\ \le\
(2/3)^{k},
\]
the value lying that close to the integer $x_{c}-x_{a}$.
\end{proposition}

\begin{proof}[Proof]
Substituting $x_{n}=\Kc(3/2)^{n}-\theta_{n}$ into Proposition~\ref{prop:circuit}
re-expresses the circuit sum in the deviations alone,
\[
W(a,k)=3^{k}\theta_{a}-2^{k}\theta_{a+k},
\]
the $\Kc$-terms cancelling because $2^{k}(3/2)^{a+k}=3^{k}(3/2)^{a}$.
The deviations survive which are now contracted by the repeated factor.
Equal factors give equal circuit sums,
$W(a,k)=W(c,k)$, so the display at $a$ and at $c$ combine into
\[
3^{k}(\theta_{c}-\theta_{a})=2^{k}(\theta_{c+k}-\theta_{a+k});
\]
the right-hand bracket has absolute value at most $1$ because $\theta\in[0,1)$,
and dividing by $3^{k}$ leaves $|\theta_{c}-\theta_{a}|\le(2/3)^{k}$.

It remains to read this at the multiplier. Rationality enters only here: since
$\Kc=\delta$, the definition $\theta_{n}=\delta(3/2)^{n}-x_{n}$ rearranges to
\[
\delta\bigl((3/2)^{c}-(3/2)^{a}\bigr)-(x_{c}-x_{a})\ =\ \theta_{c}-\theta_{a},
\]
whose right-hand side we have just bounded by $(2/3)^{k}$. So the value lies
within $(2/3)^{k}$ of the \emph{integer} $x_{c}-x_{a}$, and its distance to the
nearest integer is at most its distance to that one. (Once $k\ge2$ the bound is
below $1/2$, and $x_{c}-x_{a}$ is then the nearest integer outright; the estimate
above needs no such proviso.)
\end{proof}

This is the only place where the rationality of $\Kc$ is used in the whole chain,
and it is used only to land the value in $\Q$, where the approximation theorems
live.
For arbitrary \emph{real}
multipliers the finiteness below is false, by a Liouville construction
$\delta=\sum_{k}2^{-y_{k}}$ with $y_{k+1}\ge(1+\varepsilon)y_{k}$.

The middle display of the proof, read before rationality enters, deserves its
own statement.

\begin{corollary}[unconditional contraction]\label{cor:contract}
Let $x_{0}\ge1$. A repeated factor of length $k$ at $(a,c)$ forces
\[
\bigl|\{\Kc(3/2)^{c}\}-\{\Kc(3/2)^{a}\}\bigr|\ \le\ (2/3)^{k},
\]
with no hypothesis on $\Kc$ whatsoever.
\end{corollary}

The parity orbit cannot shadow itself for $k$ steps without dragging the
corresponding fractional parts of $\Kc(3/2)^{n}$ together to within $(2/3)^{k}$;
conversely, the repetition ceiling (Corollary~\ref{cor:slope}) caps how long the
shadowing can last. Together they are a two-way dictionary between the
combinatorics of $w$ and the distribution of $\{\Kc(3/2)^{n}\}$: every new
distributional input on $(3/2)^{n}$ modulo one converts into a complexity
statement, and every combinatorial fact about $w$ converts into a rigidity
statement about the Vijayaraghavan sequence.

\begin{definition}\label{def:viol}
For $\delta,\vartheta\in\Q$ let
\[
\mathcal V(\delta,\vartheta):=\Bigl\{(a,c)\in\N^{2}\ :\ 2\le a<c,\
\bigl\lVert\delta\bigl((3/2)^{c}-(3/2)^{a}\bigr)\bigr\rVert\le
\vartheta^{\,c}\Bigr\}.
\]
Its elements are the \emph{violators} at scale $\vartheta$: the pairs at which
the value lies closer to an integer than $\vartheta^{c}$.
\end{definition}

\begin{theorem}\label{thm:kernel}
For every rational $\delta\ne0$ and every rational $\vartheta\in(0,1)$, the set
$\mathcal V(\delta,\vartheta)$ is finite.
\end{theorem}

Theorem~\ref{thm:kernel} is what we call the \emph{kernel} of the rational half,
here and below; the word is used for it and its variants throughout, and has
nothing to do with the $k$-kernel of an automatic sequence (\S\ref{sec:intro}).
The case $\delta=1$ is \cite[Thm. 5.6]{RS26}. Opening the multiplier slot is not
free, and one of the steps that $\delta=1$ concealed breaks immediately.

\begin{lemma}[degeneracy]\label{lem:degen}
Let $\delta\in\Q$, $\delta\ne0$, and $a<c$ with $|\mathrm{num}(\delta)|\le c$.
Then $\delta((3/2)^{c}-(3/2)^{a})$ is not an integer; in particular its distance
to the nearest integer is positive.
\end{lemma}

\begin{proof}[Proof]
Write
\[
    (3/2)^{c}-(3/2)^{a}=3^{a}(3^{c-a}-2^{c-a})/2^{c},
\]
whose numerator is odd. If $\delta((3/2)^{c}-(3/2)^{a})=n\in\Z$ then
    \[
\mathrm{num}(\delta)\cdot3^{a}(3^{c-a}-2^{c-a})=n\,2^{c}\,\mathrm{den}
(\delta),
\]
    and coprimality of $2^{c}$ with the odd cofactor forces
$2^{c}\mid\mathrm{num}(\delta)$, hence $c<2^{c}\le|\mathrm{num}(\delta)|$.
\end{proof}

\begin{remark}\label{rem:degen}
At $\delta=1$ one gets $\nrm{(3/2)^{c}-(3/2)^{a}}>0$ for every pair, by
parity alone. For a general multiplier this is false, and not marginally: at
$\delta=4$, $a=1$, $c=2$ one has $4((3/2)^{2}-(3/2))=3$ exactly, an integer, to
which no Diophantine theorem applies. Lemma~\ref{lem:degen} localizes the damage
to the box $c<|\mathrm{num}(\delta)|$, which is finite and is therefore
discarded wholesale rather than by parity. This is the one place where the
multiplier costs a genuine argument.
\end{remark}

\begin{proof}[Proof of Theorem~\ref{thm:kernel}]
Discard the degeneracy box $c<|\mathrm{num}(\delta)|$, finite outright; on the
rest, Lemma~\ref{lem:degen} makes every value a non-integer, so the approximation
theorems apply. Split the remaining set by its gaps $s=c-a$.

If the gaps are bounded, it suffices to treat a fixed gap $s_{0}\ge1$. Then
    \[
        \delta((3/2)^{a+s_{0}}-(3/2)^{a})=\delta_{\mathrm{CZ}}\cdot(3/2)^{a},\quad \text{with}\quad
\delta_{\mathrm{CZ}}=\delta((3/2)^{s_{0}}-1),
\]
    and the Main Theorem of
\cite{CZ04} --- in the form derived from Theorem~\ref{thm:subspace} at $n=2$,
i.e.\ Ridout's theorem applied once --- bounds the $a$ with
    \[
        \nrm{\delta_{\mathrm{CZ}}(3/2)^{a}}<H((3/2)^{a})^{-\varepsilon} \quad\text{for}\quad
\varepsilon=\log\vartheta^{-1}/(2\log3).
\]
    Its size via
$1<|\delta_{\mathrm{CZ}}(3/2)^{a}|$ needs $(3/2)^{a}>2/|\delta|$, so a further
initial segment of $a$ is discarded; this is archimedean and harmless.

If the gaps are unbounded, choose one violator per gap. This is the delicate
selection, because the repaired pair theorem has four hypotheses and the family
must satisfy all of them, not only the visible one. With $\alpha_{1}=\delta$,
$\alpha_{2}=-\delta$ and $(u_{1},u_{2})=((3/2)^{c},(3/2)^{a})$: the archimedean
size condition $|u_{i}|\ge1$ holds because the exponents are positive; the
ratios $u_{1}/u_{2}=(3/2)^{c-a}$ are pairwise distinct --- in both orientations,
as the theorem demands --- because one violator was chosen per gap; the strict
positivity $\nrm{\delta u_{1}-\delta u_{2}}>0$, the hypothesis whose omission
broke \cite{NKR25}'s Theorem~1.3(i) (Remark~\ref{rem:nkr}), is
Lemma~\ref{lem:degen}, available because the degeneracy box was discarded first;
and the height-decay demand $\nrm{\delta u_{1}-\delta u_{2}}<
(H(u_{1})H(u_{2}))^{-\varepsilon_{1}}$ follows from the violator inequality
$\vartheta^{c}$ with $\varepsilon_{1}=\log\vartheta^{-1}/(2\log3)$, since
$H(u_{1})H(u_{2})=3^{a+c}\le3^{2c}$. The pair theorem --- derived from
Theorem~\ref{thm:subspace} at $n=3$ --- then concludes that some $(3/2)^{c}$ is
an integer. This is wrong because $2^{c}(3/2)^{c}=3^{c}$ is odd.
\end{proof}

\begin{theorem}\label{thm:t1arat}
Let $x_{0}\ge1$ and suppose $\Kc$ is rational. Then for every $C\in\N$ there is
an $m\ge1$ with $\cx(w,m)>Cm$; equivalently, $\limsup_{m}\cx(w,m)/m=\infty$.
\end{theorem}

Note: for each $C$ a single length $m$ suffices, and that is all
Cobham's bound needs negated. We do not claim the stronger
$\lim_{m}\cx(w,m)/m=\infty$, which would require every large $m$ to work.

\begin{proof}[Proof]
Write $\Kc=\delta$, nonzero since $\Kc\ge1$. Fix $C$ and pick a rational
$\vartheta\in(0,1)$ with $2/3\le\vartheta^{\,C+2}$; such a $\vartheta$ exists by
Bernoulli's inequality, taking $\vartheta=1-\frac{1/3}{C+2}$ --- the
\emph{Bernoulli scale} for $C$. This trades the
irrational scale $(2/3)^{1/(C+2)}$ for a rational one, which is what the kernel
requires. By Theorem~\ref{thm:kernel} the set $\mathcal V(\delta,\vartheta)$ is
finite, so its second coordinates are bounded by some $M$.

Now suppose $\cx(w,k)\le Ck$ for $k:=M+x_{0}+1$. Pigeonhole produces a repeated
factor: the length-$k$ factors of $w$ take values in the finite set
$\{0,1\}^{k}$, and the $Ck+1$ positions in $[2,Ck+2]$ carry at most $Ck$ distinct
factors, so two of them agree, say at $(a,c)$ with $2\le a<c\le Ck+2\le(C+2)k$.
Proposition~\ref{prop:violator} contracts it to
    \[
\nrm{\delta((3/2)^{c}-(3/2)^{a})}\le(2/3)^{k}\le\vartheta^{(C+2)k}\le
\vartheta^{c},
\]
    so $(a,c)\in\mathcal V(\delta,\vartheta)$ and $c\le M$.

But no repeated factor is that shallow. Theorem~\ref{thm:ceiling} gives
    \[
        2^{k+c}<3^{c}(x_{0}+1)\le4^{c}\cdot2^{x_{0}+1}, \quad\text{using}\  3^{c}\le4^{c}
        \ \text{and}\ x_{0}+1<2^{x_{0}+1},
    \]
        so that $k\le c+x_{0}\le M+x_{0}$, contradicting the
choice of $k$. Hence $\cx(w,k)>Ck$ for this single length $k$. The ceiling used
here is crude, and deliberately so: Corollary~\ref{cor:slope} gives the sharp
slope $0.585$ but only at $x_{0}=1$, whereas $k\le c+x_{0}$ is uniform in
$x_{0}$, which is all the reduction needs, $k\to\infty$ must force
$c\to\infty$.
\end{proof}

\begin{proof}[Proof of Theorem~\ref{thm:dich}]
If $\Kc$ is irrational we are in the second horn. Otherwise $\Kc\in\Q$ and
Theorem~\ref{thm:t1arat} gives the first.
\end{proof}

\begin{remark}\label{rem:horns}
Both horns of Theorem~\ref{thm:dich} are worth having, and neither is provable
alone here. The first would beat Theorem~\ref{thm:dublin}'s record for this word:
no linear function bounds the complexity, against a linear lower bound with a
fixed constant. The second is a conditional irrationality
statement about a constant open since 1977. As a consistency check, note that an
eventually periodic $w$ would make $\Kc$ rational \emph{and} the complexity
bounded, which Theorem~\ref{thm:aper2} excludes; so the dichotomy is sharp on
both sides.
\end{remark}

The kernel has an \emph{effective} slice, and it is worth delimiting exactly ---
both because it is all the effectivity there is, and because its boundary turns
out to be a line this paper has met before.

\begin{proposition}[Liouville floor]\label{prop:efffloor}
Let $\delta\in\Q$, $\delta\ne0$, and $a<c$ with $|\mathrm{num}(\delta)|\le c$.
Then
\[
\bigl\lVert\delta\bigl((3/2)^{c}-(3/2)^{a}\bigr)\bigr\rVert\ \ge\
\frac{1}{\mathrm{den}(\delta)\,2^{c}} .
\]
\end{proposition}

\begin{proof}[Proof sketch]
$\mathrm{den}(\delta)\,2^{c}$ clears the value to an integer, and
Lemma~\ref{lem:degen} says the value is not itself an integer; a non-integral
rational whose $D$-fold is an integer keeps distance at least $1/D$ from $\Z$.
The formal statement is the integer-certificate form
$1\le\mathrm{den}(\delta)\,2^{c}\,\lVert\cdot\rVert$.
\end{proof}

\begin{corollary}[the effective slice]\label{cor:effslice}
Let $\delta\ne0$ and $0\le\vartheta<1/2$. Then $\mathcal V(\delta,\vartheta)$ is
finite \emph{effectively}, with no Diophantine input at all: every
$(a,c)\in\mathcal V(\delta,\vartheta)$ with $c\ge|\mathrm{num}(\delta)|$
satisfies the checkable certificate
$\mathrm{den}(\delta)(2\vartheta)^{c}\ge1$, so that $c<N$ for any
$N\ge|\mathrm{num}(\delta)|$ with $\mathrm{den}(\delta)(2\vartheta)^{N}<1$.
\end{corollary}

\begin{remark}[the effective region is the empty region]\label{rem:effledger}
Three observations delimit Corollary~\ref{cor:effslice}. First, it is uniform in
the gap $c-a$: in the regime $\vartheta<1/2$ nothing is gained by fixing the
gap, and the effective statement covers the whole kernel at once --- there is no
need to isolate a bounded-gap family. Second, the slice cannot be widened by
linear forms in logarithms: the nearest integer is a \emph{free} third term ---
the relevant integer form is
$\mathrm{num}(\delta)\,3^{a}(3^{c-a}-2^{c-a})-n\,\mathrm{den}(\delta)\,2^{c}$
with $n$ unconstrained --- so a two-logarithm bound must absorb $n$ into a
composite algebraic number of height $\asymp c\log(3/2)$, and to beat the floor
of Proposition~\ref{prop:efffloor} its constant would have to be below
$\log2/(\log3\,\log(3/2))\approx1.56$, an order of magnitude beyond anything
proved. (This is the archimedean twin of the familiar fact that Baker-type
bounds lose to trivial estimates on three-term forms.) What lies beyond $1/2$ is
the Pad\'e frontier, and only for the multiplier $1$: Beukers's
$\nrm{(3/2)^{n}}\ge2^{-0.9n}$ for $n\ge5000$ \cite{Beu81}, improved by Zudilin
to $\nrm{(3/2)^{n}}\ge0.5803^{n}$ \cite{Zud07}; no comparably strong effective
bound is known with a general rational multiplier, and beyond $0.5803$ nothing
effective is known at all. Third --- the coherence --- the boundary
$\vartheta=1/2$ is the ledger of \S\ref{sub:ledger} in disguise: the floor
$\asymp2^{-c}$ beats the contraction $(2/3)^{k}$ of
Proposition~\ref{prop:violator} exactly when
$k/c>\log2/\log(3/2)=1.7095\ldots$, the demand line, while actual repeated
factors obey $k/c\le0.585$ (Corollary~\ref{cor:slope}). So the kernel is
effective precisely in the region the word never enters, and the region the
complexity application needs --- $\vartheta\to1$ in the proof of
Theorem~\ref{thm:t1arat} --- is ineffective essentially, not accidentally.
\end{remark}

\subsection{Assembling}\label{sub:compose}

\begin{proof}[Proof of Theorem~\ref{thm:main}]
Suppose $\Kc$ is algebraic. If $\Kc$ is irrational, Theorem~\ref{thm:t1a} gives
that $w$ is not automatic. If $\Kc$ is rational, Theorem~\ref{thm:t1arat} makes
the complexity of $w$ exceed every linear bound, which Cobham's bound
$\cx(a,m)=O(m)$ for automatic $a$ \cite{Cob72} forbids; so $w$ is not automatic
in that case either. The two cases partition the algebraic case. The first form
of the statement is the contrapositive: if $w$ is automatic then $\Kc$ is not
algebraic.
\end{proof}

Theorem~\ref{thm:main} strictly strengthens Corollary~\ref{cor:t1acontra}, whose
irrationality hypothesis was exactly the $\Kc\in\Q$ horn that
\S\ref{sub:kernel} removes.

\subsection{Algebraic multipliers: the Corvaja--Zannier half}\label{sub:algmult}

If the kernel held for algebraic irrational multipliers as it does for rational
ones, Theorem~\ref{thm:main} would upgrade at once to \emph{$\Kc$ algebraic
$\Rightarrow$ superlinear complexity}, single-lane, with the Mahler engine of
\S\ref{sub:af} retired and complexity content in every algebraic case. This
subsection proves the half of that extension which the printed literature
supports, shows the reduction never needed rationality, and thereby compresses
the remaining distance to a single named statement.

For real $\delta$ let $\mathcal V_{\R}(\delta,\vartheta)$ denote the violator
set of Definition~\ref{def:viol} with the distance to the nearest integer taken
in $\R$; for $\delta\in\Q$ it coincides with $\mathcal V(\delta,\vartheta)$.
Two of the three multiplier costs of \S\ref{sub:kernel} now behave
unexpectedly well. Degeneracy \emph{disappears}: for irrational $\delta$ the
value $\delta((3/2)^{c}-(3/2)^{a})$ is an irrational times a nonzero rational,
never an integer, so $\nrm{\cdot}>0$ holds for \emph{every} pair ---
Lemma~\ref{lem:degen}'s box was a rational phenomenon, and the
first casualty of the extension is in fact a simplification. And the
bounded-gap slices survive verbatim, because Theorem~\ref{thm:czalg} is printed
for algebraic $\delta$; what was a parity remark becomes a conjugate argument.

\begin{theorem}[the bounded-gap kernel, algebraic multipliers]\label{thm:algslice}
Let $\delta$ be a real algebraic irrational, let $\vartheta\in(0,1)$ be
rational, and let $S\in\N$. Then only finitely many
$(a,c)\in\mathcal V_{\R}(\delta,\vartheta)$ have gap $c-a\le S$.
\end{theorem}

\begin{proof}[Proof sketch]
Fix the gap $s_{0}\le S$ and write the value as
$\delta_{\mathrm{CZ}}\,(3/2)^{a}$ with
$\delta_{\mathrm{CZ}}=\delta((3/2)^{s_{0}}-1)$, algebraic irrational. Apply
Theorem~\ref{thm:czalg} with $q=1$, $u=(3/2)^{a}$, $H(u)=3^{a}$,
$\varepsilon=\log\vartheta^{-1}/(2\log3)$. Three clauses are discharged along
the family. The size proviso $1<|\delta_{\mathrm{CZ}}(3/2)^{a}|$ fails only on
an initial segment of $a$ (archimedean). Strict positivity holds for every
pair, by irrationality. The new obligation is the pseudo-Pisot exclusion:
$\delta_{\mathrm{CZ}}$, being algebraic irrational, has a conjugate
$z\ne\delta_{\mathrm{CZ}}$, $z\ne0$; conjugates scale along rational multiples
--- minimal polynomials obey
$p_{r\delta_{\mathrm{CZ}}}(X)=r^{\deg p}\,p_{\delta_{\mathrm{CZ}}}(X/r)$ --- so
$z\,(3/2)^{a}$ is a conjugate of the value other than the value itself, and it
leaves the unit disc once $(3/2)^{a}\ge|z|^{-1}$. Pseudo-Pisot values, like the
size proviso, are confined to an initial segment and discarded.
\end{proof}

The unbounded-gap branch of Theorem~\ref{thm:kernel} used the repaired pair
theorem --- \emph{derived} over $\Q$ from Theorem~\ref{thm:subspace} at
$n=3$. For algebraic coefficients nothing can be cited: \cite{NKR25} is
unrefereed and its Theorem~1.3(i) is false as printed (Remark~\ref{rem:nkr}),
and we know of no refereed substitute. What we can state exactly is the
obligation itself.

\begin{theorem}[conditional upgrade]\label{thm:algcond}
Let $x_{0}\ge1$ and assume the \emph{pair branch} at $\delta=\Kc$: for every
rational $\vartheta\in(0,1)$, every family
$T\subseteq\mathcal V_{\R}(\Kc,\vartheta)$ with pairwise-distinct gaps is
finite. If $\Kc$ is algebraic, then for every $C$ there is an $m\ge1$ with
$\cx(w,m)>Cm$.
\end{theorem}

\begin{proof}[Proof sketch]
If $\Kc$ is rational this is Theorem~\ref{thm:t1arat}, and the hypothesis is
not consulted. If $\Kc$ is algebraic irrational, the gap dichotomy behind
Theorem~\ref{thm:kernel} reassembles: bounded gaps fall to
Theorem~\ref{thm:algslice}, unbounded gaps to the assumed pair branch after
selecting one violator per gap, and there is no degeneracy box to discard ---
so $\mathcal V_{\R}(\Kc,\vartheta)$ is finite. Everything in the proof of
Theorem~\ref{thm:t1arat} after that finiteness then runs verbatim over $\R$:
pigeonhole, the Bernoulli scale, the contraction --- Corollary~\ref{cor:contract},
which never needed rationality --- and the growth ceiling.
\end{proof}

The pair branch is a theorem for rational $\delta$ --- supporting the proof of
Theorem~\ref{thm:kernel} ---, and for algebraic irrational $\delta$
it is a Subspace argument over the number field $\Q(\delta)$: places above
$2$, $3$, $\infty$, relative heights, $S$-enlargement at the places where
$\delta$ is not a unit. Theorem~\ref{thm:subspace} is already stated over an
arbitrary number field, so the assumption side is ready; missing is the distance between
Theorem~\ref{thm:algcond} and an unconditional upgrade.

What the proven half buys \emph{unconditionally} is a repetition statement.
Call a pair $(a,c)$ with $2\le a<c\le a+S$ a \emph{close repetition at slope
$1/L$} if the factors of $w$ of length $\lfloor c/L\rfloor+1$ at $a$ and at
$c$ coincide --- a period-$(c-a)$ stretch of length proportional to its
position, at bounded period. By the ceiling (Theorem~\ref{thm:ceiling};
$2^{2c+1}<3^{c}(x_{0}+1)$ fails for large $c$) the class is empty for $L=1$
and $c$ large, and Corollary~\ref{cor:slope}'s slope $0.585$ makes $L=2$ the
widest interesting one.

\begin{corollary}\label{cor:closerep}
Let $x_{0}\ge1$ and suppose $\Kc$ is algebraic. Then for every gap bound $S$
and every slope $1/L$, $L\ge1$, only finitely many close repetitions occur in
$w$. Contrapositively: infinitely many close repetitions at a single $(S,L)$
force $\Kc$ to be transcendental.
\end{corollary}

\begin{proof}[Proof sketch]
A close repetition is a violator of gap $\le S$ at the Bernoulli scale
$\vartheta(L)$ with $2/3\le\vartheta^{L}$: Corollary~\ref{cor:contract} gives
$\nrm{\Kc((3/2)^{c}-(3/2)^{a})}\le(2/3)^{\lfloor c/L\rfloor+1}
\le\vartheta^{c}$. For algebraic irrational $\Kc$,
Theorem~\ref{thm:algslice} finishes; for rational $\Kc$,
Theorem~\ref{thm:kernel} does.
\end{proof}

\begin{remark}[scope and lanes]\label{rem:algscope}
Under its algebraicity hypothesis, Corollary~\ref{cor:closerep} strictly
strengthens what Theorem~\ref{thm:aper2} excludes: eventual periodicity would
be one infinite family of close repetitions, while the corollary bounds every
bounded-period family of proportional-length stretches. (The two are not
comparable outright --- Theorem~\ref{thm:aper2} is unconditional.) It is also
the first repetition-structure statement available for algebraic
\emph{irrational} $\Kc$, where Theorem~\ref{thm:main} gave non-automaticity
through the Mahler lane and nothing quantitative.
Remark~\ref{rem:scope} still applies unchanged: $\Kc$ is expected transcendental,
    the trigger is expected to fail (the residue data of \S\ref{sub:outlook}
    leaves no room for close repetitions) and the content is the criterion
direction.
\end{remark}

\section{Two complements}\label{sec:comp}

\subsection{The \texorpdfstring{$2$}{2}-adic value of the word}\label{sub:padic}

Proposition~\ref{prop:partial} is an identity between rationals, and so may be
read at any place. Over $\R$ its limit is $3(\Kc-x_{0})$, an unknown real. Over
$\Qtwo$ the limit is an integer.

\begin{proposition}\label{prop:padic}
In $\Qtwo$ one has $\;\lVert 2/3\rVert_{2}=1/2$, the series $\sum_{j}w_{j}
(2/3)^{j}$ converges, and
\[
\sum_{j\ge0}w_{j}\Bigl(\tfrac23\Bigr)^{j}=-3x_{0} .
\]
More generally $\sum_{j\ge0}w_{n+j}(2/3)^{j}=-3x_{n}$ for every $n$.
\end{proposition}

\begin{proof}[Proof sketch]
The induction proving Proposition~\ref{prop:partial} uses only that $2$ and $3$
are invertible, so it runs verbatim over $\Qtwo$. Only the limit differs:
\[
\lVert x_{n}(2/3)^{n}\rVert_{2}=2^{-n}\lVert x_{n}\rVert_{2}\le2^{-n}
\longrightarrow0 ,
\]
since $x_{n}$ is a $2$-adic integer. So the orbit term vanishes and the partial
sums $3x_{n}(2/3)^{n}-3x_{0}$ tend to $-3x_{0}$. The shifted family follows by
shift-invariance of the word.
\end{proof}

\subsection{The exponent ledger}\label{sub:ledger}

A recurring proposal for this word is to periodize the digits at a repeated factor,
producing $S$-unit approximants,
and to feed those to the $p$-adic Subspace Theorem, as Adamczewski and Bugeaud
\cite{AB07} did for Loxton--van der Poorten. The proposal cannot work, and the
reason is structural.

The two constants $\log_{2}(3/2)<1<1/\log_{2}(3/2)$ are exactly the two sides of the ledger.
Corollary~\ref{cor:slope} caps the \emph{quality} $t/c$ of any repeated factor of
$w$ by $\log_{2}(3/2)=0.5849\ldots$, while periodizing at a repeated factor
yields approximants whose per-place product over $\{\infty,2,3\}$ needs
\[
t/c\ >\ (1-\lambda)/\lambda\ =\ 1.7095\ldots,\qquad
\lambda=\log(3/2)/\log3 .
\]
Since the two are reciprocal and straddle $1$, no sharpening can close the gap:
the very mechanism that makes the word complex --- the $2$-adic rigidity of
Theorem~\ref{thm:rigid}, which charges $2^{t}\mid x_{c}-x_{a}$ for a repeated
factor of length $t$ while the orbit grows only like $(3/2)^{c}$ --- is the
mechanism that caps repeated-factor quality below $1$.

\begin{proposition}[log-free form]\label{prop:notdemand}
Let $x_{0}=1$ and $a<c$. No repeated factor of length $t$ at $(a,c)$ satisfies
$17095\,c<10000\,t$.
\end{proposition}

\begin{proof}[Proof sketch]
Immediate from $41t\le24c+40$ (Corollary~\ref{cor:slope}); the two are
incompatible already for $c\ge1$.
\end{proof}

\begin{definition}\label{def:index}
For $\tau,\varepsilon>0$ and sequences $c,t\colon\N\to\N$, say $(c_{m},t_{m})$
satisfies the \emph{index condition} at $(\tau,\varepsilon)$ if
\[
(\tau+\varepsilon)\,c_{m}\log3\ \le\ t_{m}\log2
\qquad\text{for infinitely many }m .
\]
\end{definition}

This is what an application of the $p$-adic Subspace Theorem in the style of
\cite{AB07} needs from a family of $S$-unit approximants of height
$\asymp3^{c_{m}}$ and $2$-adic quality $2^{-t_{m}}$: the per-place product over
$\{\infty,2,3\}$ must beat the height to the power $-\tau-\varepsilon$, for some
$\tau>1$, infinitely often. Reading a repeated factor of length $t_{m}$ at
position $c_{m}$ as the source of such an approximant (periodize the digits
there) is the one interpretive element of this
subsection, and the theorem below is a statement about
Definition~\ref{def:index} alone.

\begin{theorem}\label{thm:noindex}
Let $(a_{m},c_{m},t_{m})$ be any family of repeated factors of $w$ (at
$x_{0}=1$) with $a_{m}<c_{m}$ and $c_{m}\to\infty$. Then $(c_{m},t_{m})$ fails
the index condition at $(\tau,\varepsilon)$, for every $\tau>1$ and every
$\varepsilon>0$.
\end{theorem}

\begin{proof}[Proof sketch]
Corollary~\ref{cor:slope} gives $t_{m}\log2<c_{m}\log3$ outright once
$c_{m}\ge3$, before $\tau$ is used at all; and $\tau+\varepsilon>1$ only makes
the demand worse. So the condition fails \emph{eventually}, which contradicts its
``infinitely often'' shape.
\end{proof}

Theorem~\ref{thm:noindex} takes the repeated-factor family as given and does not
rebuild the approximant construction: the point is that no construction can help,
whatever its details, because the exponents do not fit.

\subsection{Computations}\label{sub:outlook}

This subsection is not formally verified; it records computations and the
places where the machinery above visibly bites next. It is the only part of
the paper outside the scope of Appendix~\ref{app:lean}.

Proposition~\ref{prop:residues} makes the complexity of $w$ cheaply computable:
$\cx(w,m)$ is the number of residues modulo $2^{m}$ met by the orbit, and a
prefix of the orbit witnesses a lower bound. The result of running the first
$10^{6}$ orbit points (at $x_{0}=1$) is stark: \emph{every} residue class
modulo $2^{m}$ is met, for every $m\le16$ --- that is, the witnessed counts are
not merely superlinear but maximal, $\cx(w,m)=2^{m}$ as far as the computation
can see. (For $17\le m\le20$ the prefix is still producing new residues at its
end, so those counts are not yet saturated; nothing suggests they stop short.)
Moreover the \emph{rate} of saturation is the generic one: the position of the
last new residue --- $1471$ for $m=8$, $39346$ for $m=12$, $821306$ for $m=16$
--- matches the coupon-collector time $2^{m}(m\log2+\gamma)\approx765000$ at
$m=16$ of a uniformly random residue sequence. The orbit of $U_{3/2}$ behaves,
$2$-adically, like a random sequence. On that model full complexity is forced
--- and indeed Dubickas already conjectures it at $x_{0}=1$
\cite[p.~246]{Dub09}, judging it ``very likely as difficult as'' the
corresponding conjecture $\cx(\alpha,n)=q^{n}$ for algebraic irrationals in
integer bases. The computation above appears to be the first direct evidence,
and supports the conjecture in the generality of this paper.

\begin{conjecture}[Dubickas \cite{Dub09} at $x_{0}=1$]\label{conj:full}
For every $x_{0}\ge1$ and every $m$, $\;\cx(w,m)=2^{m}$: every binary word is a
factor of $w$.
\end{conjecture}

Two companion computations point the same way. The repeated-factor quality
$t/c$ (in the notation of \S\ref{sub:ledger}) over all pairs $a<c\le3000$ is
maximized at the early pair $(a,c,t)=(1,6,4)$ with $t/c=2/3$ --- an artifact of
the additive constant in Corollary~\ref{cor:slope} --- and restricted to
$c\ge100$ its maximum is $0.1089$, decaying with $c$: the orbit sits far below
the $0.585$ ceiling, i.e.\ the word is empirically very far from
\emph{stammering} in the sense of \cite{AB07} --- from carrying the long
near-repetitions that a Subspace attack would have to feed on.
And the truncated kernels are as large as they can be: comparing kernel
elements on their first $64$ terms, the $2$-kernel has $2^{i}$ distinct
elements at every depth $i\le12$ and the $3$-kernel $3^{i}$ at every depth
$i\le8$ --- no collapse anywhere, as non-automaticity predicts.

\section{Acknowledgements}
The author utilized Claude Code as an AI coding assistant to aid in the Lean 4
formalization of the proofs presented in this paper. The author directed and
reviewed all generated code and takes full responsibility for the mathematical
integrity and final content of the work.

\appendix
\section{Formalizing Theorem~\ref{thm:af}}\label{app:af}

Theorem~\ref{thm:af} was a cited axiom of this development until August 2026.
Removing it is the one piece of the formalization that produced findings about
the literature rather than about the word $w$, and this appendix records them.
Nothing below strengthens a statement in the body of the paper; what changed is
the trust surface, from one sequence-level black box to the two named lemmas of
\S\ref{sub:axioms}. The development is twenty-nine files and about
$11\,900$ lines devoted to \cite{AF17,AF22}, together with about $2\,700$ lines
of Mahler-system machinery on the automatic-sequence side and about $1\,500$
lines of general-purpose material (heights of tuples, regular field extensions,
rationality of power series) written for the library rather than for this proof.

\subsection*{What is proved, exactly}

The Lean statement is not the specialization of Theorem~\ref{thm:af} but a
general one: for $k$ any field of algebraic numbers, $f_{1},\dots,f_{n}\in
k[[z]]$ summing on a disc of radius $r\le1$ to functions solving a $q$-Mahler
system with polynomial matrix of nonzero determinant, and $\alpha\in k$ with
$0<|\alpha|<r$, the value $f_{i}(\alpha)$ is transcendental over $k$ or lies in
$k$. \cite[Th\'eor\`eme 1.7 (i)]{AF17} --- linear dependence over $\Qbar$
descends to linear dependence over $k$, with the support of the relation
shrinking or staying put --- is proved along the way, and is what makes the
corollary a two-line consequence.

Two differences from the printed corollary are worth naming. The field $k$ is
\emph{not} required to be a number field: the descent of \cite[\S4]{AF17} spends
one $k$-linear functional and no finiteness anywhere. Conversely, $\alpha$ is
asked to lie inside the disc of convergence, rather than merely not to be a pole
of the meromorphic continuation; recovering the printed hypothesis is exactly
what the analytic desingularization of \cite[\S5]{AF17} is for, and it is the one
thing here that is deliberately not formalized. For Theorem~\ref{thm:af} this
costs nothing, since bounded coefficients give $r=1$ and $\alpha=2/3$.

Three steps that a citation of \cite[Cor.~1.8]{AF17} would have taken on its
authority are proved instead. That an automatic sequence satisfies a $q$-Mahler
system is elementary and was already available; what is \emph{not} available from
the literature is the same system with a \emph{polynomial} matrix of nonzero
determinant --- \cite[Thm.~1]{Bec94} gives an invertible matrix over $\Qbar(z)$,
not over $\Qbar[z]$ --- and that gap is closed by an elementary rewriting of the
system rather than by a citation. The other two, that bounded coefficients force
radius of convergence $\ge1$ and hence that $|\alpha|<1$ is not a pole, and that
$k=\Q$ suffices when $\alpha$ and the coefficients are rational, are routine.

\subsection*{Where the printed proof is longer than it needs to be}

\emph{(i)} \cite[\S5]{AF17} is avoidable: the authors give two proofs of their
Th\'eor\`eme 1.7 (i), and the first, in \S4, reaches it from three elementary
lemmas and the degree-one lifting theorem, with no analytic desingularization.

\emph{(ii)} The upper-bound step of \cite[\S2.3]{AF22} --- three pages of
Taylor-coefficient bookkeeping: Cauchy--Hadamard, a binomial re-expansion at
$\alpha^{q^{k_{0}}}$, a separate estimate for the coefficients of the auxiliary
matrix, a convolution --- is not needed. The quantity being bounded is a single
value, so the maximum modulus principle delivers it from a supremum on one fixed
circle, and the whole of their Lemma 2.13 collapses to ``a power series with a
zero of order $p$ is small at a small argument''. This was the step the project
had budgeted as its largest risk.

\emph{(iii)} \cite[Lemma 2.2]{AF22} --- the eventual linearity in $\delta_{2}$ of
a dimension $d(\delta_{1},\delta_{2})$, not to be confused with
Theorem~\ref{thm:af17reg} above --- is proved in their Appendix A.1 by clearing
denominators in a basis of an annihilator and counting linear forms in a dual
space. None of that is needed: the dimensions are those of a filtration
generated by iterating multiplication by $z$, and a five-line monotonicity gives
\emph{exact} eventual linearity, which is stronger than what is asked for.

\emph{(iv)} The Puiseux field of \cite{AF22} is a container, not a tool. The word
occurs exactly twice in that paper, both times in the definition of the ambient
field; no valuation, order of vanishing, or fractional exponent appears anywhere
in \S2. Any algebraically closed extension of $\Qbar(z)$ in which the relation
matrix is algebraic serves the same purpose.

\emph{(v)} No shrinking discs, and no matrix inverse. \cite[Rem.~2.10]{AF22}
shrinks the disc from step to step because $A(z^{q^{j}})$ acquires poles; a
polynomial Mahler matrix has none, and every system here is polynomial. And the
normalizing matrix $a=A_{k_{0}}(\alpha)\varphi(\xi)^{-1}$ --- where
$\xi:=\alpha^{q^{k_{0}}}$ and $A_{k}(z):=A(z)A(z^{q})\cdots A(z^{q^{k-1}})$ is
the iterated Mahler matrix --- is used only through the identity
$a\varphi(\xi)=A_{k_{0}}(\alpha)$, so no inverse is ever formed.

\emph{(vi)} The primitive element of \cite[\S2.4]{AF22} can be dispensed with
altogether, and this is also the fix for \emph{(viii)} below. The coefficients of
the relation generate a finitely generated torsion-free $\Qbar[z]$-module, hence
a free one; a basis expresses each of them with \emph{polynomial} coordinates, so
there is no denominator to clear and no condition on $\alpha$ to arrange.

\subsection*{One false claim, and one impossible order}

\emph{(vii)} \cite{AF22} justify the equality of two dimensions in the
auxiliary-function step by ``$P(Y,z)\mapsto P(MY,z)$ is an automorphism of
$\Qbar[Y,z]_{\delta_{1},\delta_{2}}$''. The two dimensions are those of
$\mathcal{I}(\delta_{1},\delta_{2})$ and of $\mathcal{J}(\delta_{1},\delta_{2})$,
the bidegree-bounded pieces of their ideal $\mathcal{I}\subset\Qbar(z)[Y]$ of
relations and of the second ideal $\mathcal{J}:=\{P:P(MY,z)\in\mathcal{I}\}$ they
introduce for the twist. The quoted map is not an automorphism. Their
$\delta_{1}$ bounds the degree in \emph{each} matrix variable $y_{ij}$
separately, and a linear substitution does not preserve that: already for $m=2$,
$y_{11}y_{21}\mapsto(M_{11}y_{11}+M_{12}y_{21})(M_{21}y_{11}+M_{22}y_{21})$ has
degree $2$ in $y_{11}$.

What belongs in its place is an inclusion rather than an automorphism, and it
asks nothing of $M$:
\[
  \text{for every }M\in\mathrm{M}_{m}(\Qbar),\qquad
  P\in\Qbar[Y,z]_{\delta_{1},\delta_{2}}
  \ \Longrightarrow\
  P(MY,z)\in\Qbar[Y,z]_{m^{2}\delta_{1},\delta_{2}}.
\]
Two invariants prove it, and they are what a formalization has to supply in
place of the printed sentence. The substitution sends each $y_{ij}$ to a linear
form, so it does not raise the \emph{total} degree in $Y$; and its coefficients
are constants, so it keeps every coefficient inside the $\Qbar$-span of
$1,z,\dots,z^{\delta_{2}}$. The factor $m^{2}$ is the price of a round trip
between the two notions of degree: all $m^{2}$ separate degrees $\le\delta_{1}$
force total degree $\le m^{2}\delta_{1}$, and total degree $\le m^{2}\delta_{1}$
forces each separate degree $\le m^{2}\delta_{1}$ again. (It could be sharpened
to $m$, since $Y\mapsto MY$ acts within each column of $Y$, so only the $m$
variables in the column of $y_{ij}$ can raise its degree; nothing below needs
the sharper form.)

The dimension count is then repaired with no further change, because inflating
the first index is cheap: iterating their Lemma 2.3,
$d(2\delta_{1},\delta_{2})\le2^{m^{2}}d(\delta_{1},\delta_{2})$, shows that a
factor $c$ in the first index costs at most $(2^{m^{2}})^{\lceil\log_{2}c\rceil}$.
The index the auxiliary-function step actually needs is $(m^{2}+1)\delta_{1}$ ---
the $m^{2}\delta_{1}$ of the inclusion above, plus the $\delta_{1}$ contributed by
the powers of $F$ --- so what replaces their equality of dimensions is
\[
  d\bigl((m^{2}+1)\delta_{1},\,\delta_{2}\bigr)
  \ \le\ 2^{\,m^{2}\lceil\log_{2}(m^{2}+1)\rceil}\;d(\delta_{1},\delta_{2}),
\]
a constant depending on $m$ alone. That is why the repair is free: the step
drives $\delta_{1}\to\infty$ with $m$ fixed, so a factor independent of
$\delta_{1}$ changes nothing. A formalizer may also drop $\mathcal{J}$
altogether, as is done here --- apply $M$ to the auxiliary function rather than
to the ideal, and every dimension is measured against $\mathcal{I}$, so their
Lemmas 2.2 and 2.3 are proved once instead of twice.

\emph{(viii)} \cite{AF22} let $d(z)$ be a common denominator of the matrices of
their decomposition (2.36) and then assume $k_{0}$ chosen large enough that
$d(\alpha^{q^{k_{0}}})\ne0$. This is legitimate for them, because their $d$ is
read off the untwisted matrix and is fixed \emph{before} $k_{0}$ is chosen. It
cannot be imitated in the order the argument must actually be run: \S2.4 never
evaluates $\varphi$ at $\alpha^{q^{k_{0}}}$, it evaluates the twisted matrix
$\varphi(z^{q^{k_{0}}})$ at $\alpha$, so the decomposition is of the twisted
matrix, its $d$ depends on $k_{0}$, and there is no second ``for
$k_{0}\gg1$'' left to spend. Nor can the untwisted degree be used instead:
$\varphi(z^{q^{k_{0}}})$ may have \emph{strictly smaller} degree than $\varphi$
--- for $\varphi=\sqrt{z}$ and $q$ even it is rational. The free-module argument
of \emph{(vi)} removes the denominator, and with it the difficulty.

\subsection*{Hypotheses that are load-bearing and invisible until written down}

Characteristic zero in Theorem~\ref{thm:af17reg}, as recorded in
\S\ref{sub:axioms}: over $\mathbb{F}_{p}(z)$ the statement is false, so an axiom
stated for an arbitrary field would have been a false axiom.

Algebraic closedness of the base field in Theorem~\ref{thm:af22branch}\emph{(c)}.
It is invisible in the printed statement and becomes visible the moment one asks
where the normalizing matrix lives: the value of an algebraic function at an
algebraic point is algebraic, and lands in the image of the base field only when
that field is algebraically closed --- true of \cite{AF22}'s $\Qbar$, and needed
because the auxiliary-function step substitutes $Y\mapsto MY$ with $M$ a matrix
over that field.

The algebraicity in Theorem~\ref{thm:af22branch} must be required of the
\emph{entries} of $\varphi$, not of the ambient field, and this is the difference
between a true statement and a false one. The ambient field is relatively
algebraically closed, so for every $\xi$ it contains $\sqrt{z-\xi}$, which has no
analytic branch at $\xi$; a hypothesis asserting that the ambient field is
realized by analytic functions would be refutable. Only finitely many elements
can be realized at once --- which is why the interface realizes a subring.

$\alpha\ne0$, $|\alpha|<1$ and $q\ge2$ are what make ``for $k\gg1$'' meaningful:
the points $\alpha^{q^{k}}$ then have strictly decreasing moduli, so they are
pairwise distinct and all but finitely many avoid any fixed finite bad set. At
$q=1$ the sequence is constant and the statement is false at a bad point. The
same $q\ge2$ is load-bearing a second time, in the descent below.

\subsection*{The two-field obstruction}

Theorem~\ref{thm:af22branch} wants an algebraically closed base field. The
lower-bound step is Liouville's inequality and wants a number field, because the
height machinery is available only for fields carrying a finite family of
archimedean absolute values --- and that is a fact about $\Qbar$, not an accident
of packaging. No field is both. \cite{AF22} never meet this, because their
heights are absolute rather than relative.

The resolution is a descent to a single number field, fixed once and never
enlarged. It must be a single one: the heights available are relative, $[F:\Q]$
times the absolute ones, so letting each auxiliary polynomial choose its own
field would let the Liouville constant grow with the degree parameter and destroy
the contradiction that the whole argument is aimed at. Finitely many numbers
generate it --- the coefficients of $A$, the coefficients of the given relation,
the entries of the branch value, the constant terms $f_{i}(0)$, and $\alpha$ ---
and the point worth recording is that the solutions contribute only their
\emph{constant} terms: comparing coefficients in $f(z)=A(z)f(z^{q})$ expresses
$[z^{n}]f_{i}$ through coefficients of index $<n$, an induction that exists only
because $q\ge2$.

\subsection*{What the library lacked}

Liouville's inequality, which turns out not to be a product-formula computation
at all --- the product formula is already spent inside the height of an inverse
--- and costs fifteen lines. A sum bound for the projective height of tuples: the
naive analogue of the one-coordinate bound is \emph{false}, since the projective
height is scaling-invariant and a sum is not (over $\Q$, $H(c,0)=H(0,1)=1$ for
every $c\ne0$, while $H\bigl((c,0)+(0,1)\bigr)=H(c)$ is unbounded), and what a
matrix product actually needs is the bound for sums whose terms are all read off
one tuple. A sharp evaluation bound in which each degree is charged once, the
naive one carrying a factor equal to the number of monomials --- exponential in
the degree parameter, and therefore worthless where the constant must be
independent of it. ``Regular extension $\Rightarrow$ linearly disjoint'', against
an arbitrary finite-dimensional subextension and not merely a simple one. And,
last, the Mahler substitution itself on the ambient field: it has to be built
from the substitution on power series, extended to the field of formal Laurent
series and then to its algebraic closure, and the final step needs two algebra
structures on one type at once --- which in Lean is what a type synonym is for.

\subsection*{Interfaces}

Two structural facts came out of the analytic layer and are worth stating
independently of Lean. The realization of algebraic functions by analytic ones
cannot be a homomorphism defined on the whole ambient field: $z-c$ is a unit
there, and its realization vanishes at $c$. So it is defined on a subring, and
then its injectivity is no longer free --- that subring is never a field, for the
same reason --- and has to be part of what is assumed rather than derived. And
taking germs along the principal filter of the domain, rather than at the point,
makes a germ an honest function and evaluation at a point an honest ring
homomorphism; this is what turns \cite{AF22}'s ``well-defined for $k\gg1$'', the
one genuinely awkward phrase in \S2.2, into a total operation, and makes their
radicality lemma a five-line consequence of ``the germ ring of a reduced ring is
reduced''.

\subsection*{What remains cited}

Theorems~\ref{thm:af17reg} and~\ref{thm:af22branch}, and nothing else:
\texttt{\#print axioms} on the Lean form of Theorem~\ref{thm:af}, and on every
result of \S\ref{sub:af} that consumes it, returns those two together with the
ambient axioms of classical logic. The reasons they are cited are different in
kind. Theorem~\ref{thm:af17reg} is cited because its own proof descends into the
analytic theory of Mahler functions --- natural boundaries, non-polar
singularities --- which is a project rather than a step; the deduction
\cite{AF22} make \emph{from} it is not, and is proved here. Theorem
\ref{thm:af22branch} is cited because none of the three ingredients of its short
literature proof is available in the formalized library: no vanishing criterion
for resultants, hence no ``finitely many singularities''; no field of convergent
power series and no Newton--Puiseux theorem, hence no way to produce a branch;
and no analytic implicit function theorem for algebraic function elements. A
model is exhibited for its clauses \emph{(b)} and \emph{(c)}, so that the cited
form is known to be satisfiable rather than vacuous; the further clauses folded
into it, as described in \S\ref{sub:axioms}, are argued from \cite{AF22}'s
construction and are not separately modelled.

\section{Formalization}\label{app:lean}

\begingroup\raggedright
Every statement above was verified in Lean~4, except the four cited axioms,
which are assumed, and the computations of \S\ref{sub:outlook}, which are not
formalized.
See \url{https://github.com/rwst/Transcendence-Criteria}.
The files live under \texttt{RB/}, with the twenty-nine files
\texttt{CITED/AdamczewskiFaverjon*.lean} (Appendix~\ref{app:af}) and with
\texttt{CITED/SubspaceTheorem.lean}, \texttt{CITED/AlloucheShallitBasic.lean},
\texttt{CITED/AlloucheShallitComplexity.lean} and
\texttt{CITED/CorvajaZannierAlgebraic.lean}
supplying inputs. ``std3'' abbreviates the ambient axioms of classical Lean
(propositional extensionality, choice, quotient soundness). Four cited axioms
occur, each marked below: ``[AF17]'' for \texttt{AF.lemme\_2\_2}, ``[AF22]'' for
\texttt{AF.lemma\_2\_8}, ``[Sub]'' for
\texttt{Subspace.evertseSchlickewei}, and ``[CZ04a]'' for
\texttt{CZ.pseudoPisot\_approx\_alg} (the algebraic-multiplier instance of
Theorem~\ref{thm:czalg}, carried only by \S\ref{sub:algmult}; the rational
instance is derived from [Sub] and is not an axiom). Every entry is fully
proved in Lean except the four marked \emph{cited axiom}. All footprints below
were confirmed with \texttt{\#print axioms}.
\par\endgroup

\medskip
\noindent\emph{Cited inputs $(\S\ref{sub:axioms})$.}
\begin{center}
\tiny
\renewcommand{\arraystretch}{1.25}
\setlength{\tabcolsep}{2.5pt}
\fitwidth{%
\begin{tabular}{@{}llll@{}}
\hline
Statement & Lean identifier & File & Status\\
\hline
Thm.~\ref{thm:af17reg} & \texttt{AF.lemme\_2\_2} & \texttt{AdamczewskiFaverjonRegular} & cited axiom [AF17]\\
Thm.~\ref{thm:af22branch} & \texttt{AF.lemma\_2\_8} & \texttt{AdamczewskiFaverjonBranchFull} & cited axiom [AF22]\\
Thm.~\ref{thm:af} & \texttt{AF.transcendental\_or\_rat\_of\_automatic} & \texttt{AdamczewskiFaverjon} & std3 + [AF17] + [AF22]\\
\quad\emph{[AF22] Thm 2.1, the lifting thm} & \texttt{AF.theoreme\_2\_1} & \texttt{AdamczewskiFaverjonTheoreme21} & std3 + [AF17] + [AF22]\\
\quad\emph{[AF17] Thm 1.7 (i)} & \texttt{AF.theoreme\_1\_7\_i} & \texttt{AdamczewskiFaverjonTheoreme17} & std3 + [AF17] + [AF22]\\
\quad\emph{[AF17] Cor 1.8, general form} & \texttt{AF.corollaire\_1\_8} & \texttt{AdamczewskiFaverjonTheoreme17} & std3 + [AF17] + [AF22]\\
\quad\emph{nonsingular Mahler system} & \texttt{RB.exists\_nonsingular\_mahlerSystem} & \texttt{MahlerNonsingular} & std3\\
\quad\emph{radius $\ge1$, not a pole} & \texttt{RB.one\_le\_genFunSeries\_radius} & \texttt{MahlerAnalytic} & std3\\
Thm.~\ref{thm:subspace} & \texttt{Subspace.evertseSchlickewei} & \texttt{SubspaceTheorem} & cited axiom [Sch91]\\
--- & \texttt{CZ.pseudoPisot\_approx\_of\_subspace} & \texttt{CorvajaZannierProof} & std3 + [Sub]\\
--- & \texttt{NKR.sUnit\_pair\_integrality\_of\_subspace} & \texttt{NairKumarRoutProof} & std3 + [Sub]\\
Rem.~\ref{rem:nkr} & \texttt{NKR.thm13i\_unrepaired\_false} & \texttt{NairKumarRout} & std3\\
--- & \texttt{AS.kKernel}, \texttt{AS.IsAutomatic} & \texttt{AlloucheShallitBasic} & def\\
--- & \texttt{AS.complexity} & \texttt{AlloucheShallitComplexity} & def\\
Rem.~\ref{rem:linear} & \texttt{AS.complexity\_linear\_of\_automatic} & \texttt{AlloucheShallitComplexity} & std3\\
--- & \texttt{AS.not\_automatic\_of\_complexity\_superlinear} & \texttt{AlloucheShallitComplexity} & std3\\
\hline
\end{tabular}}
\end{center}

\medskip
\noindent\emph{The orbit, the word, the constant, and rigidity
$(\S\ref{sec:basic}$--$\S\ref{sec:rig})$.}
\begin{center}
\tiny
\renewcommand{\arraystretch}{1.25}
\setlength{\tabcolsep}{2.5pt}
\fitwidth{%
\begin{tabular}{@{}llll@{}}
\hline
Statement & Lean identifier & File & Status\\
\hline
Def.~\ref{def:orbit} & \texttt{RB.x}, \texttt{RB.wmin} & \texttt{Basic} & def\\
Prop.~\ref{prop:sel} & \texttt{two\_mul\_x\_succ}, \texttt{wmin\_le\_one} & \texttt{Basic} & std3\\
Def.~\ref{def:W} & \texttt{RB.W} & \texttt{Basic} & def\\
Prop.~\ref{prop:circuit} & \texttt{circuit\_sum} & \texttt{Basic} & std3\\
Prop.~\ref{prop:partial} & \texttt{x\_mul\_pow} & \texttt{Basic} & std3\\
Def.~\ref{def:K} & \texttt{RB.K} & \texttt{Basic} & def\\
Prop.~\ref{prop:Klim} & \texttt{tendsto\_x\_mul\_pow}, \texttt{tsum\_wmin\_eq} & \texttt{Basic} & std3\\
Prop.~\ref{prop:window} & \texttt{x\_le\_K\_mul\_pow}, \texttt{K\_mul\_pow\_le} & \texttt{Basic} & std3\\
Def.~\ref{def:rep} & \texttt{RB.IsRepetition} & \texttt{Rigidity} & def\\
Thm.~\ref{thm:rigid} & \texttt{isRepetition\_iff\_dvd} & \texttt{Rigidity} & std3\\
\quad\emph{equal circuit sums} & \texttt{lemmaR\_int} & \texttt{Rigidity} & std3\\
Thm.~\ref{thm:ceiling} & \texttt{repetition\_pow\_lt} & \texttt{Rigidity} & std3\\
\quad\emph{orbit growth} & \texttt{two\_pow\_mul\_x\_add\_one\_le} & \texttt{Basic} & std3\\
Cor.~\ref{cor:slope} & \texttt{repetition\_linear\_bound} & \texttt{Rigidity} & std3\\
Prop.~\ref{prop:residues} & \texttt{complexity\_eq\_ncard\_residues} & \texttt{Residues} & std3\\
\quad\emph{factors are residues} & \texttt{factor\_eq\_iff\_residue\_eq} & \texttt{Residues} & std3\\
Cor.~\ref{cor:dubfloor} & \texttt{complexity\_lower\_bound} & \texttt{Residues} & std3\\
Prop.~\ref{prop:mod3} & \texttt{x\_succ\_mod\_three} & \texttt{Residues} & std3\\
Thm.~\ref{thm:aper}, \ref{thm:aper2} & \texttt{not\_eventually\_periodic} & \texttt{Rigidity} & std3\\
Thm.~\ref{thm:cf} & \texttt{closed\_form} & \texttt{ClosedForm} & std3\\
\hline
\end{tabular}}
\end{center}

\medskip
\noindent\emph{Transcendence criteria $(\S\ref{sec:main})$.}
\begin{center}
\tiny
\renewcommand{\arraystretch}{1.25}
\setlength{\tabcolsep}{2.5pt}
\fitwidth{%
\begin{tabular}{@{}llll@{}}
\hline
Statement & Lean identifier & File & Status\\
\hline
Thm.~\ref{thm:t1a} & \texttt{not\_automatic\_of\_K\_algebraic\_irrational} & \texttt{NotAutomatic} & std3 + [AF17] + [AF22]\\
Cor.~\ref{cor:t1acontra} & \texttt{transcendental\_of\_automatic\_of\_irrational} & \texttt{NotAutomatic} & std3 + [AF17] + [AF22]\\
$\theta_{n}$ (\S\ref{sub:kernel}) & \texttt{RB.tail}, \texttt{tail\_nonneg}, \texttt{tail\_lt\_one} & \texttt{OrbitKernel} & std3\\
$\theta_{n}=\{\Kc(3/2)^{n}\}$ & \texttt{tail\_eq\_fract} & \texttt{Residues} & std3\\
Prop.~\ref{prop:violator} & \texttt{dist\_le\_of\_repetition} & \texttt{OrbitKernel} & std3\\
Cor.~\ref{cor:contract} & \texttt{abs\_tail\_sub\_le\_of\_repetition} & \texttt{OrbitKernel} & std3\\
Def.~\ref{def:viol} & \texttt{RB.scaledViolators} & \texttt{ScaledKernel} & def\\
Lem.~\ref{lem:degen} & \texttt{dist\_pos\_of\_num\_le} & \texttt{ScaledKernel} & std3\\
Thm.~\ref{thm:kernel} & \texttt{scaledViolators\_finite} & \texttt{ScaledKernel} & std3 + [Sub]\\
Prop.~\ref{prop:efffloor} & \texttt{one\_le\_den\_mul\_two\_pow\_mul\_dist} & \texttt{EffectiveSlice} & std3\\
Cor.~\ref{cor:effslice} & \texttt{scaledViolators\_finite\_of\_lt\_half} & \texttt{EffectiveSlice} & std3\\
Thm.~\ref{thm:czalg} & \texttt{CZ.pseudoPisot\_approx\_alg} & \texttt{CorvajaZannierAlgebraic} & cited axiom\\
$\mathcal V_{\R}$ (\S\ref{sub:algmult}) & \texttt{RB.algViolators} & \texttt{AlgebraicKernel} & def\\
Thm.~\ref{thm:algslice} & \texttt{algGapBounded\_slice\_finite} & \texttt{AlgebraicKernel} & std3 + [CZ04a]\\
Thm.~\ref{thm:algcond} & \texttt{superlinear\_of\_K\_algebraic\_of\_pairBranch} & \texttt{AlgebraicKernel} & std3 + [CZ04a] + [Sub]\\
Cor.~\ref{cor:closerep} & \texttt{closeRepetitions\_finite\_of\_K\_algebraic} & \texttt{AlgebraicKernel} & std3 + [CZ04a] + [Sub]\\
Thm.~\ref{thm:t1arat} & \texttt{superlinear\_of\_K\_rat} & \texttt{RationalK} & std3 + [Sub]\\
Thm.~\ref{thm:main} & \texttt{not\_automatic\_of\_K\_algebraic} & \texttt{RationalK} & std3 + [AF17] + [AF22] + [Sub]\\
Thm.~\ref{thm:main} & \texttt{transcendental\_of\_automatic} & \texttt{RationalK} & std3 + [AF17] + [AF22] + [Sub]\\
\quad\emph{rational case} & \texttt{not\_automatic\_of\_K\_rat} & \texttt{RationalK} & std3 + [Sub]\\
Thm.~\ref{thm:dich} & \texttt{superlinear\_or\_K\_irrational} & \texttt{RationalK} & std3 + [Sub]\\
\hline
\end{tabular}}
\end{center}

\medskip
\noindent\emph{The complements $(\S\ref{sec:comp})$.}
\begin{center}
\tiny
\renewcommand{\arraystretch}{1.25}
\setlength{\tabcolsep}{2.5pt}
\fitwidth{%
\begin{tabular}{@{}llll@{}}
\hline
Statement & Lean identifier & File & Status\\
\hline
Prop.~\ref{prop:padic} & \texttt{norm\_two\_thirds}, \texttt{x\_mul\_pow\_padic} & \texttt{PadicValue} & std3\\
Prop.~\ref{prop:padic} & \texttt{tsum\_wmin\_padic}, \texttt{tsum\_wmin\_shift\_padic} & \texttt{PadicValue} & std3\\
Prop.~\ref{prop:notdemand} & \texttt{not\_demand} & \texttt{NoStammeringRoute} & std3\\
Thm.~\ref{thm:noindex} & \texttt{not\_indexConditionExpFreq} & \texttt{NoStammeringRoute} & std3\\
\hline
\end{tabular}}
\end{center}

\medskip
The sanity witnesses $x=1,2,3,5,8,12,18,27,41$ and $w=1,0,1,1,0,0$ for $x_{0}=1$
(matching A061419 and the first letters of $g_{3/2}$) and the closed-form values
$\lfloor \Kc(3/2)^{0}\rfloor=1$, $\lfloor \Kc(3/2)^{4}\rfloor=8$, $\lfloor
\Kc(3/2)^{8}\rfloor=41$ are verified by decision procedures
(\texttt{x\_sanity}, \texttt{wmin\_sanity}, \texttt{closed\_form\_sanity}).


\end{document}